\documentclass[12pt, a4paper, dvipsnames]{amsart}
\usepackage{amsmath,amsthm}
\usepackage[margin=3cm]{geometry}
\usepackage[active]{srcltx}
\usepackage{amssymb}
\usepackage{latexsym}
\usepackage{hyperref}
\usepackage[all,cmtip,arrow,2cell]{xy}
\usepackage{enumitem}

\usepackage{tikz-cd}
\usepackage{tikz-3dplot}
\usepackage{pgfplots}
\pgfplotsset{compat=1.16}
\usetikzlibrary{intersections}
\usetikzlibrary{positioning}
\usepgfplotslibrary{fillbetween}

\usepackage{cite}

\usepackage{eucal}

\theoremstyle{plain}
\newtheorem{Theorem}{Theorem}[section]
\newtheorem{Proposition}[Theorem]{Proposition}
\newtheorem{Lemma}[Theorem]{Lemma}

\theoremstyle{definition}
\newtheorem{Remark}[Theorem]{Remark}
\newtheorem{Definition}[Theorem]{Definition}

\newtheorem{Example}[Theorem]{Example}

\newtheorem{Notation}[Theorem]{Notation}

\theoremstyle{remark}

\DeclareMathOperator{\Diff}{Diff}

\DeclareMathOperator{\Ric}{Ric}
\DeclareMathOperator{\Riem}{Riem}

\newcommand{\bbN}{\mathbb{N}}

\newcommand{\bbR}{\mathbb{R}}
\newcommand{\bbZ}{\mathbb{Z}}

\newcommand{\calF}{\mathcal{F}}

\newcommand{\calX}{\mathcal{X}}

\newcommand{\frakg}{\mathfrak{g}}

\newcommand{\Volg}{{\mathrm{vol}_g}}
\newcommand{\id}{\mathrm{id}}

\newcommand{\Lie}{\mathcal{L}}
\newcommand{\Div}{\mathrm{div}}

\newcommand{\Dtot}{\mathbf{d}}

\newcommand{\calLor}{{\mathcal{L}\mathrm{or}}}
\newcommand{\Lor}{{\mathrm{Lor}}}

\renewcommand{\phi}{\varphi}
\renewcommand{\epsilon}{\varepsilon}

\newcommand{\Hide}[1]{}
\renewcommand{\Hide}[1]{{\color{blue}#1}}

\begin{document}

\title[The homotopy momentum map of general relativity]
{The homotopy momentum map of\\
general relativity}

\author[C.~Blohmann]{Christian Blohmann}
\address{Max-Planck-Institut f\"ur Mathematik, Vivatsgasse 7, 53111 Bonn, Germany}
\email{blohmann@mpim-bonn.mpg.de}

\subjclass[2020]{53D20; 18N40, 37K06, 83C05}

\date{\today}

\keywords{general relativity, multisymplectic geometry, homotopy momentum map, lagrangian field theory, variational bicomplex, diffeomorphism symmetry}

\begin{abstract}
We show that the action of spacetime vector fields on the variational bicomplex of general relativity has a homotopy momentum map that extends the map from vector fields to conserved currents given by Noether's first theorem to a morphism of $L_\infty$-algebras.
\end{abstract} 
\maketitle



\section{Introduction}

\subsection{Motivation}

The diffeomorphism symmetry of general relativity, a mathematical implementation of the Einstein equivalence principle, is one of its defining features. In contrast to the internal symmetry of gauge theories, diffeomorphisms are external symmetries since they act not only on the fields (i.e.~lorentzian metrics), but also on spacetime. The initial value problem, which yields the hamiltonian formulation of the field dynamics, and the presymplectic structure on the space of fields, which yields the Poisson bracket of observables, both depend on the choice of a codimension 1 submanifold as initial time-slice. But such a submanifold is not invariant under diffeomorphisms. In physics terminology: it breaks the symmetry. The consequence is that the basic ingredients of quantization, the hamiltonian and the Poisson bracket, are not compatible with the diffeomorphism symmetry. This issue lies at the heart of some of the fundamental open problems in general relativity and has captivated the interest of many authors since the 1960s.

One of its mathematical symptoms is that the action of the group of diffeomorphisms and the action of the Lie algebra of vector fields are not hamiltonian. More precisely, Noether's first theorem, which associates to a symmetry a conserved momentum does not define a homomorphism of Lie algebras. (The Noether momenta are the components of the Einstein tensor integrated over the codimension 1 submanifold.) Worse, the space of Noether momenta is not even closed under the Poisson bracket.

In an earlier paper we could show that there is a natural diffeological groupoid describing the choices of initial submanifolds, which exhibits the Poisson brackets as the bracket of its Lie algebroid \cite{BlohmannFernandesWeinstein:2013}. Next, we have developed a notion of hamiltonian Lie algebroids, which generalizes the notion of hamiltonian Lie algebra action to the setting of Lie algebroids \cite{BlohmannWeinstein:2018}. We have conjectured that the Noether charges of general relativity are the components of the momentum section of a hamiltonian Lie algebroid, which would give a conceptual explanation of some of the intriguing features of the constraint functions. Finally, in \cite{BlohmannSchiavinaWeinstein:2022} we have interpreted the momenta as elements of a generalized Lie-Rinehart algebra, which is connected to the BV-BFV approach to boundary conditions in classical field theories.

In this paper, we sidestep the choice of initial submanifolds altogether by using higher algebraic structures. We show that the map from vector fields to their Noether currents is part of a homotopy momentum map in the sense of multisymplectic geometry. 

\subsection{Content and main results}

In Sec.~\ref{sec:LFT} we study the premultisymplectic form $\omega = EL + \delta \gamma$ of a lagrangian field theory (LFT), where $EL$ is the Euler-Lagrange form and $\gamma$ a boundary form. We prove in Prop.~\ref{prop:HamCurr} that the obstruction of a premultisymplectic vector field $X$ to be hamiltonian lies in bidegree $(0,n-1)$ of the variational bicomplex, where $n$ is the dimension of the spacetime manifold. This shows that the $L_\infty$-algebra of hamiltonian forms can be interpreted as generalized current algebra. We introduce the notion of manifest diffeomorphism symmetry (Def.~\ref{def:ManifestDiffSym}) and observe that every such symmetry has a hamiltonian momentum map that is given explicitly in terms of the lagrangian and the boundary form (Prop.~\ref{prop:ManifestMomMap}).

In Sec.~\ref{sec:VarBicomplex} we consider the case of general relativity. Generalizing the concept of tensor fields, we introduce the notion of covariant and contravariant families of forms in the variational bicomplex (Def.~\ref{def:CovariantForms}). We then show that the product of families, the contraction of indices, the raising and lowering of indices, etc.~have properties that are analogous to tensor fields. In Sec.~\ref{sec:CovDerForms} we introduce the notion of covariant derivative of families of forms. In Props.~\ref{prop:DivFormula1} and~\ref{prop:DivFormula2} we derive divergence formulas that show that horizontally exact forms can be expressed as the contraction of families of forms with the covariant derivative.

Sec.~\ref{sec:GRmomentum} contains the main results. In Thm.~\ref{thm:LepageInvariant} we prove that the Lepage form $L + \gamma$, which is the primitive of the premultisymplectic form $\omega$, is invariant under the action of spacetime vector fields. This implies that the action has a homotopy momentum map, which is described explicitly in Thm.~\ref{thm:GRHomMomentum}.

\subsection{Relation to previous work}
\label{sec:PreviousWork}

The study of multisymplectic forms in classical field theory goes back to at least the 1970s. Best known is perhaps the highly influential, but never finished GiMmsy project (named by the initials of the collaborators involved, with the main protagonists Gotay and Marsden capitalized). Their goal was ``to explore some of the connections between initial value constraints and gauge transformations'' in classical field theories with constraints, such as general relativity \cite[p.~1]{GiMmsy1}. Towards this end they introduced the notion of multimomentum maps \cite[p.~46]{GiMmsy1} (see also \cite[Sec.~4.2]{CarinenaCrampinIbort:1991}). Given the action $\rho: \frakg \to \calX(M)$ of a Lie algebra on a manifold $M$ with a closed $(n+1)$-form $\omega$, a multimomentum map was defined as a smooth map $M \to \frakg^* \otimes \wedge^{n-1} T^* M$ or, equivalently, a linear map $f: \frakg \to \Omega^{n-1}(M)$, such that $d\,f(a) = - \iota_{\rho(a)} \omega$ for all $a \in \frakg$.

Missing from this definition was a requirement of compatibility with the Lie bracket of $\frakg$, analogous to hamiltonian momentum maps in symplectic geometry. It seems natural to require $f$ to be $\frakg$-equivariant. Alternatively, the hamiltonian forms in $\Omega^{n-1}(M)$ can be equipped with a ``Poisson bracket'' and $f$ required to commute with the brackets. However, both conditions turn out to be too strong and rarely satisfied. Moreover, the ``Poisson bracket'' does not satisfy the Jacobi identity, so that it is not immediately clear for what kind of algebraic structure $f$ should be a homomorphism.

This lack of compatibility of algebraic structures leads to issues in the study of the constraints of classical field theories with diffeomorphism symmetries, of which general relativity is the theory ``par excellence'' \cite[Interlude I, p.~52]{GiMmsy1}. The constraint functions of general relativity are given by the values of the multimomentum map integrated over the Cauchy surface. The resulting map is called the energy-momentum map \cite[Sec.~7B]{GiMmsy2}. While the energy-momentum map shows that the constraint functions arise from the multimomentum map, it does not explain the relation between the Lie brackets of the symmetry algebra of vector fields and the Poisson brackets of the constraints. (For the history of this often studied but elusive problem see Sec.~4 of \cite{BlohmannFernandesWeinstein:2013}.) From the viewpoint of homotopical algebra, this was to be expected: Lie brackets that satisfy the Jacobi identity up to exact terms and maps that preserve the brackets up to exact terms are generally not compatible with homotopies of the underlying complexes. Instead we have to use the homotopy algebraic structure, i.e.\ the minimal extension of Lie algebras to differential complexes that is stable under quasi-isomorphisms. For a better behaved theory of multimomentum maps, we are thus led to $L_\infty$-algebras.

In \cite[Thm.~5.2]{Rogers:2012} it was shown that the bracket on hamiltonian forms in $\Omega^{n-1}(M)$ has a natural extension to an $L_\infty$-algebra structure on a graded subspace of the de Rham complex, with the 1-bracket given by the de Rham differential.\footnote{In \cite{BarnichEtAl:1998} it was shown that the bracket of integrated local functions on the jet bundle has extensions to alternative $L_\infty$-algebra structures on cohomological resolutions. However, these $L_\infty$-structures were not given by an explicit construction, depended on choices, and did not suggest a stronger notion of multimomentum maps.} It was realized in \cite[Def./Prop.~5.1]{CalliesFregierRogersZambon:2016} that this is the natural setting for the generalization of \emph{hamiltonian} momentum maps to the multisymplectic setting, defined as morphisms $\mu: \frakg \to L_\infty(M,\omega)$ of $L_\infty$-algebras. The $\mu_1$ component of every homotopy momentum map is a multimomentum map \cite[Sec.~12.1]{CalliesFregierRogersZambon:2016}. For the obstructions to lifting a multimomentum map to a homotopy momentum map see \cite[Sec.~9.2]{CalliesFregierRogersZambon:2016}.

In local lagrangian field theories, a multimomentum map is given by Noether's theorem \cite[Sec.~4.1]{BridgesEtAl:2010}. Finding a homotopy momentum map, however, is a much more difficult problem, even more so in general relativity, where the Hilbert-Einstein lagrangian is non-polynomial in the fields and of second jet order. The situation simplifies greatly if the multi(pre)symplectic form has a primitive, $\omega = d\lambda$, and if the action leaves $\lambda$ invariant. Then a homotopy momentum map can be defined by inserting the fundamental vector fields in $\lambda$ \cite[Sec.~8.1]{CalliesFregierRogersZambon:2016}. If we want to check whether this applies to a classical field theory, we have to identify the correct $\lambda$ as well as the correct action of the diffeomorphism group. Many authors use for $\lambda$ the boundary 1-form, so that $\omega$ is the universal current in the sense of \cite{Zuckerman:1986}, and assume that the action is vertical (e.g.~\cite{BridgesEtAl:2010}). We will show that instead we have to take the sum of the lagrangian and the boundary 1-form for $\lambda$ and the diagonal action~\eqref{eq:DiffActionGR} on fields and spacetime by vertical and horizontal vector fields.

\subsection{Conventions}
\label{sec:Conventions}

The spacetime manifold $M$ is assumed to be smooth, finite-dimensional, and second countable. The infinite jet manifold $J^\infty F$ of a smooth fibre bundle $F \to M$ is viewed as pro-manifold, so that its de Rham complex, i.e.~the variational bicomplex, is an ind-differential complex. For the computations and proofs in this paper, however, this will not play much of a role. The same goes for the diffeological structure on the space of fields $\calF = \Gamma(M,F)$, of which we will only use the fact that the diffeological tangent bundle is given by the space of sections of the vertical tangent bundle $T\calF = \Gamma(M,VF)$. For the grading and differentials of the variational bicomplex we use the notation of \cite{DeligneFreed:1999}: A form in $\Omega^{p,q}(J^\infty F)$ has vertical degree $p$ and horizontal degree $q$. The vertical differential is denoted by $\delta$ and the horizontal differential by $d$. Otherwise, we follow \cite{Anderson:1989}, with specific references given in the text. We use the summation convention throughout the paper, so that all repeated indices are being summed over.

\begin{Remark}
\label{rmk:Momentum}
Instead of ``momentum map'', many authors use the term ``moment map'', which derives from a mistranslation of the French term ``moment'' as in ``moment cin\'etique'' (angular momentum) or ``application moment'' \cite{Souriau:1970}.
\end{Remark}

\subsection{Brief review of homotopy momentum maps}
\label{sec:BriefReviewHomMom}

For the reader's convenience, we give a brief review of the main notions of multisymplectic geometry used in this paper. This is also necessary to fix the notation, the choice of gradings, and the signs.

Let $M$ be a manifold with a closed $(n+1)$-form $\omega$. A pair $(X,\alpha)$ consisting of a vector field $X \in \calX(M)$ and a form $\alpha \in \Omega^{n-1}(M)$ is called \textbf{hamiltonian} if
\begin{equation*}
  \iota_X \omega = - d\alpha \,.
\end{equation*}
A vector field or a form is called hamiltonian if it is part of a hamiltonian pair. We denote the space of hamiltonian vector fields by $\calX_\mathrm{ham}(M)$ and the space of hamiltonian forms by $\Omega^{n-1}_\mathrm{ham}(M)$.

For the pair $(M,\omega)$ we can construct an $L_\infty$-algebra $L_\infty(M,\omega)$ defined as follows \cite[Thm.~5.2]{Rogers:2012}. The $\bbZ$-graded vector space is
\begin{equation*}
  L_\infty(M,\omega)_i = 
  \begin{cases}
    \Omega^{n-1}_\mathrm{ham}(M) &; i = 0 \\
    \Omega^{n-1 + i}(M) &; 1-n \leq i < 0 \\
    0 & ; \text{otherwise} \,.
  \end{cases} 
\end{equation*}
The brackets $l_k: \wedge^k L_\infty(M,\omega) \to L_\infty(M,\omega)$ are defined by
\begin{equation*}
  l_1 (\alpha_1) = d\alpha_1 
\end{equation*}
for $\deg \alpha_1 < 0$, by
\begin{equation*}
\begin{split}
  l_k(\alpha_1 \wedge \ldots \wedge \alpha_k) 
  &= -(-1)^{\tfrac{1}{2}k(k+1)} 
    \iota_{X_k} \cdots \iota_{X_2} \iota_{X_1} \omega
  \\
  &= -(-1)^{k} 
    \iota_{X_1} \iota_{X_2} \cdots  \iota_{X_k} \omega  
  \,,     
\end{split}
\end{equation*}
for $k>1$ and $\deg \alpha_1 = \ldots = \deg \alpha_k = 0$ 
where $(X_i,\alpha_i)$ are hamiltonian pairs, and by zero in all other cases. With this degree convention, the degree of $l_k$ is $2-k$.

\begin{Definition}[Def./Prop.~5.1 in \cite{CalliesFregierRogersZambon:2016}]
Let $M$ be a manifold with a closed $(n+1)$-form $\omega$. Let $\rho: \frakg \to \calX(M)$ be a homomorphism of Lie algebras. A \textbf{homotopy momentum map} of the action $\rho$ is a homomorphism of $L_\infty$-algebras
\begin{equation*}
  \mu: \frakg \longrightarrow L_\infty(M,\omega)
  \,,
\end{equation*}
such that
\begin{equation*}
  \iota_{\rho(a)} \omega = - d\, \mu_1(a)
\end{equation*}
for all $a \in \frakg$.
\end{Definition}

We recall that a morphism $\mu: L' \to L$ of $L_\infty$-algebras is given by a family of linear maps $\mu_k: \wedge^k L' \to L$, $k \geq 1$ of degree $1-k$, subject to relations that are best expressed either in terms of the $L_\infty$-operad or in the language of formal pointed manifolds. If the domain $L' = \frakg$ is a Lie algebra, as is the case for a homotopy momentum map, the conditions for $\mu$ to be a morphism simplify greatly. They are best expressed in terms of the boundary operator $\delta: \wedge^\bullet \frakg \to \wedge^{\bullet - 1}\frakg$ of the Chevalley-Eilenberg complex for Lie homology,
\begin{equation*}
  \delta(a_1 \wedge \ldots \wedge a_k) =
  \sum_{1 \leq i<j \leq k} (-1)^{i+j}
  [a_i, a_j] \wedge a_1 \wedge \ldots \hat{a}_i \ldots
  \hat{a}_j \wedge \ldots \wedge a_k
  \,.
\end{equation*}
A collection of linear maps $\mu_k: \wedge^k \frakg \to L_\infty(M,\omega)$ is a homotopy momentum map if and only if \cite[Prop.~3.8]{CalliesFregierRogersZambon:2016}
\begin{equation*}
 d \mu_k(a_1 \wedge \ldots \wedge a_k) 
 + \mu_{k-1}\delta(a_1 \wedge \ldots \wedge a_k)
 = (-1)^{\tfrac{1}{2}k(k+1)} \iota_{\rho(a_k)} \cdots \iota_{\rho(a_1)} \omega
 \,,
\end{equation*}
for all $1 \leq k \leq n$, where we set $\mu_{0} := 0$ and $\mu_{n + 1} := 0$. This relation can be interpreted homotopically as follows. Shifting the degree of $\frakg$ by $1$ and shifting the degree of the de Rham complex by $n-1$, the right hand side can be expressed in terms of the degree 0 map
\begin{equation*}
\begin{aligned}
  \nu: S(\frakg[1])
  &\longrightarrow \Omega(M)[n+1]
  \\
  \nu(a_1 \wedge \cdots \wedge a_k)
  &:=
  (-1)^{k} \iota_{\rho(a_1)} \cdots \iota_{\rho(a_k)} \omega
  \,,
\end{aligned}
\end{equation*}
where we have used that $S(\frakg[1])_{-k} \cong \wedge^k \frakg$. The maps $\mu_k$ have degree $-1$. The condition for $\mu$ to be a morphism of $L_\infty$-algebras can be written succinctly as  \cite[Sec.~6.2]{CalliesFregierRogersZambon:2016}
\begin{equation}
\label{eq:muNullHomotopy}
  d\mu_k + \mu_{k-1} \delta = \nu
  \,,
\end{equation}
that is, a homotopy momentum map $\mu$ is a null-homotopy of the map of cochain complexes $\nu$. 

In degree $k=1$ the condition~\eqref{eq:muNullHomotopy} reads $d\mu_1(a_1) = - \iota_{\rho(a_1)} \omega$, that is $\bigl( \rho(a_1), \mu_1(a_1) \bigr)$ is a hamiltonian pair. With this relation, $\nu$ can be expressed in terms of the $L_\infty$-brackets as
\begin{equation*}
  \nu(a_1 \wedge \ldots \wedge a_k)
  = - l_k\bigl( \mu_1(a_1) \wedge \ldots \wedge \mu_1(a_k) \bigr)
  \,.
\end{equation*}
For $k=2$, Eq.~\eqref{eq:muNullHomotopy} is spelled out as
\begin{equation}
\label{eq:muHomLie}
  l_2\bigl( \mu_1(a_1) \wedge \mu_1(a_2) \bigr)
  =
  \mu_1\bigl( [a_1, a_2] \bigr) 
  - d\mu_2(a_1 \wedge a_2)
  \,,
\end{equation}
which shows that the failure of $\mu_1$ to be a homomorphism of Lie algebras is a $d$-exact term. For $k=3$ we obtain
\begin{align*}
  l_3\bigl( \mu_1(a_1) \wedge \mu_1(a_2) \wedge \mu_1(a_3) \bigr)
  &=
  \mu_2(
      [a_1, a_2] \wedge a_3 
    + [a_2, a_3] \wedge a_1 
    + [a_3, a_1] \wedge a_2 )
  \\
  &{}\quad
  - d\mu_3(a_1 \wedge a_2 \wedge a_3)
  \,.
\end{align*}

\begin{Proposition}[Sec.~8.1 in \cite{CalliesFregierRogersZambon:2016}]
\label{prop:ExactMomentum}
Let $\omega = d\lambda$ for some $\lambda \in \Omega^n(M)$. If $\lambda$ is invariant under the action $\rho: \frakg \to \calX(M)$, i.e.
\begin{equation*}
  \Lie_{\rho(a)} \lambda = 0
\end{equation*}
for all $a \in \frakg$, then it has a homotopy momentum map $\mu: \frakg \to L_\infty(M,\omega)$ given by
\begin{equation*}
  \mu_k(a_1 \wedge \ldots \wedge a_k)
  = \iota_{\rho(a_1)} \cdots \iota_{\rho(a_k)} \lambda
  \,.
\end{equation*}
\end{Proposition}

\begin{Notation}
For shorter notation we will write the $k$-bracket also as
\begin{equation*}
\begin{split}
  l_k(\alpha_1 \wedge \ldots \wedge \alpha_k)
  &\equiv l_k(\alpha_1, \ldots, \alpha_k)
  \\
  &\equiv \{\alpha_1, \ldots, \alpha_k\}
\end{split}
\end{equation*}
Analogously, we will write the momentum map as
\begin{equation*}
  \mu_k(a_1 \wedge \ldots \wedge a_k)
  \equiv \mu_k(a_1, \ldots, a_k)
  \,.
\end{equation*}
In this notation, Eq.~\eqref{eq:muHomLie} is written as
\begin{equation}
\label{eq:muBracket}
  \{ \mu_1(a_1), \mu_1(a_2) \}
  =
  \mu_1\bigl( [a_1, a_2] \bigr) 
  - d\mu_2(a_1, a_2)
  \,.
\end{equation}
\end{Notation}

\section{Multisymplectic geometry of lagrangian field theories}
\label{sec:LFT}

The space of fields of a field theory is the set of sections $\calF = \Gamma(M,F)$ of a fibre bundle over a manifold $M$, naturally equipped with the functional diffeology. The lagrangian is a map $\tilde{L}: \calF \to \Omega^n(M)$, where $n$ is the dimension of $M$. We will assume that the lagrangian is local, i.e.~a differential operator, so that $\tilde{L}(\phi) = (j^\infty \phi)^* L$, where $L \in \Omega^{0,n}(J^\infty F)$ is the lagrangian form and $j^\infty \phi: M \to J^\infty F$ is the infinite jet prolongation of the field $\phi: M \to F$. 

If $M$ is compact, we can define the action $S: \calF \to \bbR$ by $S(\phi) = \int_M \tilde{L}(\phi)$. Many interesting and important lorentzian spacetimes are not compact, however, so that the action is generally not well-defined. Therefore, we have to formulate the action principle, the derivation of the field equations,  the notion of symmetries, etc.~in a cohomological form within the variational bicomplex \cite{Zuckerman:1986,DeligneFreed:1999}.

In Sec.~\ref{sec:CohomAct} we state the action principle in its cohomological form, essentially replacing integration by taking cohomology classes with respect to the spacetime differential $d$. By the cohomological version of partial integration the variation of the lagrangian can be written as $\delta L = EL - d\gamma$, where $EL$ is the Euler-Lagrange form and the $\gamma$ the boundary form  \cite{Zuckerman:1986}. $EL$ can be viewed as the differential operator of the field equations, so that it governs the dynamics of the field theory. The integration of $\delta\gamma$ over a codimension 1 submanifold of spacetime yields a presymplectic form on the space of fields, so that it describes the Poisson brackets of the field observables.

For the multisymplectic approach we will consider $\omega = EL + \delta \gamma$, which is an exact $(n+1)$-form of degree $(n+1)$. Its primitive is the Lepage form $L + \gamma$. In Prop.~\ref{prop:HamCurr} we show that the bidegree $(0,n-1)$-component of a hamiltonian form is a conserved current in the sense of \cite{Zuckerman:1986}. The $L_\infty$-algebra associated to $\omega$ as in \cite{Rogers:2012} can therefore be viewed as a higher version of the current algebra.

If $M$ is closed, so that the action $S: \calF \to \bbR$ is defined, a symmetry of the LFT is an automorphism $\Phi$ of $\calF$ that leaves the action invariant, $\Phi^* S = S$. Infinitesimally, a vector field $\Xi$ on $\calF$ is a symmetry if $\Lie_\Xi S = 0$. This is the case if and only if $\Lie_\Xi \tilde{L}$ is $d$-exact, which we take as the general definition of a symmetry. For a local lagrangian we have to require that the vector field $\Xi$, too, is local. In Sec.~\ref{sec:NoetherSym} we observe that a vector field on the diffeological space $\calF$ is local if and only if it descends to a vector field on $J^\infty F$, which is the infinite prolongation of an evolutionary ``vector field''. Such vector fields are strictly vertical in the sense that their inner derivative commutes with the horizontal differential. The strictly horizontal vector fields whose inner derivative commutes with the vertical differential are the lifts of the spacetime vector fields by the Cartan connection. 

In Def.~\ref{def:ManifestDiffSym} we introduce the notion of manifest diffeomorphism symmetry, which is an action $\rho: \calX(M) \to \calX(J^\infty F)$ of the Lie algebra of spacetime vector fields on the infinite jet bundle, such that $\rho(v) = \xi_v + \hat{v}$ is the sum of a strictly vertical vector field $\xi_v$ and the Cartan lift of $v$ that leaves the Lepage form invariant, $\Lie_{\rho(v)} (L + \gamma) = 0$. We point out in Prop.~\ref{prop:ManifestMomMap} that such a symmetry has a homotopy momentum map given by inserting the fundamental vector fields of the action into the Lepage form, which is a special case of the well-known Prop.~\ref{prop:ExactMomentum}.

\subsection{The cohomological action principle}
\label{sec:CohomAct}

A variety of ingredients can play a constitutive role in the mathematical study of classical field theories. For the purpose of this paper the following minimal definition will suffice:

\begin{Definition}
\label{def:LFT}
A \textbf{local lagrangian field theory} (LFT) consists of a manifold $M$ of dimension $n$, called the spacetime, a fibre bundle $F \to M$, called the configuration bundle, and a form $L \in \Omega^{0,n}(J^\infty F)$ in the variational bicomplex, called the lagrangian. 
\end{Definition}

Let $\alpha \in \Omega^{p,q}(J^\infty F)$, where $p$ denotes the vertical and $q$ the horizontal degree. The vertical differential will be denoted by $\delta$, the horizontal differential by $d$. The form $\alpha$ is represented by a form on a finite dimensional jet manifold $J^k F$, which is given by a map $\tilde{\alpha}: J^k F \to \wedge^{p+q} T^* J^k F$. In this way $\alpha$ gives rise to a $k$-th order differential operator
\begin{align*}
  D_\alpha: \calF &\longrightarrow 
  \Gamma(M, \wedge^{p+q} T^* J^k F)
  \\
  \phi &\longmapsto \tilde{\alpha} \circ j^k \phi
  \,,
\end{align*}
where $j^k \phi: F \to J^k F$ is the $k$-th jet prolongation of $\phi$. In this notation, the integrand of the action can be written as $\tilde{L}(\phi) = D_L(\phi)$. 

The target of the differential operator $D_\alpha$ is not a vector space, so it does not make sense to consider the equation ``$D_\alpha(\phi) = 0$'', even though this is how the corresponding PDE is often written. And even if $F \to M$ is a vector bundle so that $\wedge^{p+q} T^* J^k F \to F$ is a vector bundle, the right hand side cannot be required to be the zero section, as this would imply that $\phi$ is the zero section of $F \to M$. Instead, we have to use that there is a zero form $0 \in \Omega^{p,q}(J^k F)$ in every bidegree. The PDE can then be written more carefully as $D_\alpha (\phi) = D_0(\phi)$. If this equation holds, we will say that $\alpha$ \textbf{vanishes at $\phi \in \calF$}.

\begin{Definition}
\label{def:CohActPrinc}
A form $\alpha \in \Omega^{p,q}(J^\infty F)$ is said to be \textbf{$d$-exact at $\phi \in \calF$} if there is a form $\beta \in \Omega^{p,q-1}(J^\infty F)$ such that $\alpha - d\beta$ vanishes at $\phi$.
\end{Definition}

\begin{Definition}
A field $\phi \in \calF$ is said to satisfy the \textbf{cohomological action principle} if $\delta L$ is $d$-exact at $\phi$.
\end{Definition}

If a form $\alpha$ is of top horizontal degree $q=n$, then there is a unique representative $P\alpha$ of its $d$-cohomology class, which has the following property: The form $\alpha$ is $d$-exact at $\phi$ if and only if $P\alpha$ vanishes at $\phi$. The map $P: \Omega^{p,n}(J^\infty F) \to \Omega^{p,n}(J^\infty F)$ is the cohomological version of partial integration and straightforward to compute. It is called the interior Euler operator. Forms in the image of $P$ are called ``source'' for $p=1$ and ``functional'' for $p > 1$ \cite{Takens:1981}, \cite[Def.~2.5 and Ch.~3]{Anderson:1989}.

The map $E := P\delta: \Omega^{p,n}(J^\infty F) \to \Omega^{p+1,n}(J^\infty F)$ is called the Euler operator. The source form $EL \in \Omega^{1,n}(J^\infty F)$ is called the \textbf{Euler-Lagrange form}. Since the interior Euler operator does not change the $d$-cohomology class, $EL = P\delta L$ and $\delta L$ are in the same $d$-cohomology class, i.e.
\begin{equation*}
  EL - \delta L = d\gamma \,,
\end{equation*}
for some $\gamma \in \Omega^{1,n-1}(J^\infty F)$. The form $\gamma$ is called a \textbf{boundary form}.

From the properties of source forms it follows that a field satisfies the cohomological action principle if and only if it satisfies the Euler-Lagrange equation
\begin{equation*}
  D_{EL}(\phi) = D_0(\phi)
  \,.
\end{equation*}
In physics terminology, such a field is called \textbf{on shell}.

\subsection{Premultisymplectic structure and current \texorpdfstring{$L_\infty$}{L-infinity}-algebra}
\label{sec:CurrentAlgebra}

Let $\gamma$ be a boundary form. The form
\begin{equation}
\label{eq:Lepage}
  \lambda := L + \gamma
\end{equation}
of total degree $n$ will be called the \textbf{Lepage form}\footnote{For the terminology see \cite{Krupka:1983} or \cite[p.~199]{Anderson:1989}. Deligne and Freed call $\lambda$ the ``total Lagrangian'' \cite[p.~161]{DeligneFreed:1999}.}. Let the total differential of $J^\infty F$ be denoted by $\Dtot = \delta + d$. The total differential of the Lepage form is
\begin{equation}
\label{eq:MultiSympForm}
\begin{split}
  \omega &:= \Dtot\lambda \\
  &= EL + \delta \gamma
  \,,
\end{split}
\end{equation}
which is the premultisymplectic structure we are interested in.

On $J^\infty F$ we have the splitting of vector fields into a vertical and horizontal component which leads to the bigrading on the de Rham complex. Moreover, we have the acyclicity theorem of the variational bicomplex. This leads to the following description of hamiltonian vector fields.


\begin{Proposition}
\label{prop:HamCurr}
Let $X$ be a vector field on $J^\infty F$ with vertical component $X^\perp$. Then $X$ is hamiltonian with respect to the premultisymplectic form $\omega = EL + \delta\gamma$ if and only if
\begin{itemize}

\item[(i)] $\Lie_{X} \omega = 0$, and

\item[(ii)] $\iota_{X^{\!\perp}} EL = dj$ for some $j \in \Omega^{0,n-1}(J^\infty F)$.

\end{itemize}
\end{Proposition}

In the proof, we will use the following lemma \cite[Thm.~11.1.6]{Delgado:PhD}.

\begin{Lemma}
\label{lem:AcycCurrent}
Let $\beta \in \Omega^n(J^\infty F)$ be a $\Dtot$-closed form and
\begin{equation*}
  \beta = \beta_0 + \ldots + \beta_n
\end{equation*}
its decomposition into summands of bidegree $\deg \beta_k = (k,n-k)$. Then $\beta$ is $\Dtot$-exact if and only if $\beta_0$ is $d$-exact.
\end{Lemma}
\begin{proof}
A form $\alpha \in \Omega^{n-1}(J^\infty F)$ can be decomposed as
\begin{equation*}
 \alpha 
 = \alpha_0 + \ldots + \alpha_{n-1}
 \,,
\end{equation*}
into components of bidegree $\deg \alpha_k = (k,n-1-k)$. The total differential decomposes as
\begin{equation*}
 \Dtot\alpha 
 = d\alpha_0 + (\delta\alpha_0 + d\alpha_1) + \ldots
 + (\delta\alpha_{n-2} + d\alpha_{n-1}) + \delta\alpha_{n-1}
 \,,
\end{equation*}
into summands of homogeneous bidegree, where the first summand has bidegree $(0,n)$ and the last $(n,0)$. Assume that $\beta = \Dtot \alpha$. This condition must hold in each bidegree individually. In particular we have $\beta_0 = d\alpha_0$.

Conversely, assume that $\beta_0 = d\alpha_0$ for some $\alpha_0 \in \Omega^{0,n-1}(J^\infty F)$. The total differential of $\beta$ decomposes as
\begin{equation*}
 \Dtot\beta
 = (\delta\beta_0 + d\beta_1) + \ldots
 + (\delta\beta_{n-1} + d\beta_{n}) + \delta\beta_{n}
\end{equation*}
into summands of homogeneous bidegree, where the first summand has bidegree $(1,n)$ and the last $(n+1,0)$. By assumption $\Dtot \beta = 0$, which has to hold in each bidegree separately,
\begin{align*}
  0 &= \delta\beta_0 + d\beta_1
  \\
  0 &= \delta\beta_1 + d\beta_2
  \\
  &{}\qquad\vdots
  \\
  0 &= \delta\beta_{n-1} + d\beta_n
  \\
  0 &= \delta\beta_n
  \,.
\end{align*}
From the first equation we get
\begin{equation*}
\begin{split}
  0 
  &= \delta \beta_0 + d \beta_1 = \delta d\alpha_0 + d\beta_1
  \\
  &= d(-\delta \alpha_0 + \beta_1) \,.
\end{split}
\end{equation*}
It follows from the acyclicity theorem for the variational bicomplex \cite[Thm.~4.6]{Takens:1981} that $-\delta\alpha_0 + \beta_1 = d\alpha_1$ for some $\alpha_1 \in \Omega^{1,n-2}(J^\infty F)$. The bidegree $(2,n-1)$ component of $\Dtot \beta = 0$ can now be written as
\begin{equation*}
\begin{split}
  0 
  &= \delta \beta_1 + d \beta_2 
  = \delta(\delta \alpha_0 + d\alpha_1) + d\beta_2
  \\
  &= d(-\delta \alpha_1 + \beta_2) \,.
\end{split}
\end{equation*}
As before, it follows from the acyclicity theorem that $\beta_2 = \delta\alpha_1 + d\alpha_2$ for some $\alpha_2 \in \Omega^{2,n-3}(J^\infty F)$. By induction, we obtain forms $\alpha_0, \ldots, \alpha_{n-1}$ such that $\Dtot \alpha = \beta$ for $\alpha =  \alpha_0 + \ldots + \alpha_{n-1}$.
\end{proof}

\begin{proof}[Proof of Prop.~\ref{prop:HamCurr}]
Let $X^\perp$ be the vertical and $X^\parallel$ the horizontal component of $X$. Assume that $\iota_X \omega = - \Dtot \alpha$. The left hand side decomposes as
\begin{equation*}
\begin{split}
  \iota_X \omega 
  &= (\iota_{X^{\!\perp}} + \iota_{X^\parallel})
     (EL + \delta \gamma)
  \\
  &=
  \iota_{X^{\!\perp}} EL 
  + \bigl(\iota_{X^\parallel} EL 
  + \iota_{X^{\!\perp}} \delta\gamma \bigr)
  + \iota_{X^\parallel} \delta \gamma
  \,,
\end{split}
\end{equation*}
into summands of bidegree $(0,n)$, $(1,n-1)$, and $(2,n-2)$. We conclude that the bidegree $(0,n)$ component of the hamiltonian condition is $\iota_{X^{\!\perp}} EL = - d\alpha_0$, which is condition (ii) for $\alpha_0 = -j$. Since $\omega$ is closed, we have $\Lie_X \omega = \Dtot \iota_X \omega = 0$ which is condition (i). 

Conversely, assume that (i) and (ii) hold. This means that $-\beta = \iota_X \omega$ is $\Dtot$-closed and that $\beta_0 = d\alpha_0$, where $\alpha_0 = -j$. It now follows from Lem.~\ref{lem:AcycCurrent} that $\iota_X \omega = -\beta = -\Dtot \alpha$.
\end{proof}

A form $j \in \Omega^{0,n-1}(J^\infty F)$ is also called a \textbf{current}. A current is called \textbf{conserved} if it is $d$-closed on shell, i.e.~if $dj$ vanishes at every solution of the Euler-Lagrange equation. Prop.~\ref{prop:HamCurr} shows that the degree $(0,n-1)$ component of a hamiltonian form is a conserved current. In this sense, $L_\infty(J^\infty F, \omega)$ can be viewed as higher current algebra.

The $k$-bracket of hamiltonian forms $\alpha_1, \ldots, \alpha_k$ is given by
\begin{equation*}
\begin{split}
  \{\alpha_1, \ldots, \alpha_k\}
  &= -(-1)^k 
    (\iota_{X_1^{\!\perp}} + \iota_{X_1^\parallel}) \cdots
    (\iota_{X_k^{\!\perp}} + \iota_{X_k^\parallel})
    (EL + \delta\gamma)
  \\ 
  &= 
  \sum_{1\leq i < j \leq k} (-1)^{k-i-j}
  (\iota_{X_1^\parallel} \cdots 
  \widehat{\iota_{X_i^\parallel}} \cdots 
  \widehat{\iota_{X_j^\parallel}} \cdots 
  \iota_{X_k^\parallel})
  (\iota_{X_i^{\!\perp}}\iota_{X_j^{\!\perp}}\delta\gamma)
  \\
  &{}\quad
  + \sum_{1\leq i \leq k} (-1)^{i-1}
  (\iota_{X_1^\parallel} \cdots 
  \widehat{\iota_{X_i^\parallel}} \cdots 
  \iota_{X_k^\parallel})
  (\iota_{X_i^{\!\perp}} EL)
  \\
  &{}\quad
  + \sum_{1\leq i \leq k} (-1)^{i-1}
  (\iota_{X_1^\parallel} \cdots 
  \widehat{\iota_{X_i^\parallel}} \cdots 
  \iota_{X_k^\parallel})
  (\iota_{X_i^{\!\perp}} \delta\gamma)
  -(-1)^k  
  (\iota_{X_1^\parallel} \cdots 
  \iota_{X_k^\parallel}) EL
  \\
  &{}\quad
  -(-1)^k 
  (\iota_{X_1^\parallel} \cdots 
  \iota_{X_k^\parallel}) \delta\gamma
  \,,
\end{split}
\end{equation*}
where $X_1, \ldots, X_k$ are the hamiltonian vector fields. The 2-bracket is given by
\begin{equation}
\label{eq:2bracket}
\begin{split}
  \{ \alpha, \beta \}
  &= (\iota_{Y^{\!\perp}} + \iota_{Y^\parallel})
     (\iota_{X^{\!\perp}} + \iota_{X^\parallel})
     (EL + \delta\gamma)
  \\
  &= \iota_{Y^{\!\perp}} \iota_{X^{\!\perp}} \delta\gamma
  + ( \iota_{Y^\parallel} \iota_{X^{\!\perp}} 
     -\iota_{X^\parallel} \iota_{Y^{\!\perp}} ) EL
  \\
  &{}\quad
  + (\iota_{Y^\parallel} \iota_{X^{\!\perp}} 
    -\iota_{X^\parallel} \iota_{Y^{\!\perp}}) \delta\gamma
  + \iota_{X^\parallel} \iota_{X^\parallel} EL
  \\
  &{}\quad
  + \iota_{Y^\parallel} \iota_{X^\parallel} \delta\gamma
  \,,
\end{split}    
\end{equation}
where $(X,\alpha)$ and $(Y,\beta)$ are hamiltonian pairs.

\subsection{Noether symmetries}
\label{sec:NoetherSym}

The diffeological tangent space of $\calF$ is given by the space of sections of the vertical tangent bundle, $T\calF \cong \Gamma(M,VF)$, so that a vector field on $\calF$ is given by a map $\Xi: \Gamma(M,F) = \calF \to T\calF = \Gamma(M,VF)$. This map is called local if it is a differential operator, i.e.~if there is a commutative diagram
\begin{equation*}
\begin{tikzcd}[column sep=large]
\Gamma(M,F) \times M \ar[d, "j^k"] \ar[r, "\Xi \times \id_M"] &
\Gamma(M,VF) \times M \ar[d, "j^0"]
\\
J^k F \ar[r, "\xi_0"] & VF
\end{tikzcd}
\end{equation*}
where $VF = \ker (TF \to TM)$ is the vertical vector bundle. The map $\xi_0$ is often called an evolutionary ``vector field''.

\begin{Remark}
\label{rmk:Evolutionary}
We put quotes around evolutionary ``vector field'' because it cannot be naturally viewed as an actual vector field unless the configuration bundle $F \to M$ is equipped with a flat connection. Readers who are used to this traditional (but abusive) terminology (e.g.~Def.~1.15 in \cite{Anderson:1989}) are kindly asked to ignore the quotes.
\end{Remark}

The map $\xi_0$ can be prolonged to a vector field $\xi$ on $J^\infty F$, which is the unique vector field such that 
\begin{equation}
\label{eq:VecFieldLift}
\begin{tikzcd}[column sep=large]
\Gamma(M,F) \times M \ar[d, "j^\infty"] \ar[r, "\Xi \times \id_M"] &
\Gamma(M,VF) \times M \ar[d, "Tj^\infty"]
\\
J^\infty F \ar[r, "\xi"] & TJ^\infty F
\end{tikzcd}
\end{equation}
commutes. If such a commutative diagram exists, we will say that $\xi \in \calX(J^\infty F)$ lifts to the vector field $\Xi \in \calX(\calF)$. The following proposition is a purely algebraic characterization of such vector fields.

\begin{Proposition}
\label{prop:StrictVert}
A vector field $\xi \in \calX(J^\infty F)$ lifts to a vector field on $\calF$ if and only if $[\iota_\xi, d] = 0$.
\end{Proposition}
\begin{proof}
The proof follows from a straightforward computation in jet coordinates.
\end{proof}

The kernel of $\Omega^{1,0}(J^\infty F)$ defines an integrable distribution on $TJ^\infty F$, called the Cartan distribution, which can be interpreted as a flat Ehresmann connection on $TJ^\infty F \to M$. The horizontal lift of a vector field $v \in \calX(M)$ is denoted by $\hat{v} \in \calX(J^\infty F)$. Since the connection is flat, the map $\calX(M) \to \calX(J^\infty F)$, $v \mapsto \hat{v}$ is a homomorphism of Lie algebras. The following is a purely algebraic characterization of such lifts.

\begin{Proposition}
\label{prop:StrictHor}
A vector field $\xi$ on $J^\infty F$ is the horizontal lift of a vector field on $M$ by the Cartan connection if and only if $[\iota_{\xi}, \delta] = 0$.
\end{Proposition}
\begin{proof}
The proof follows from a straightforward computation in jet coordinates.
\end{proof}

The last two propositions can be understood geometrically as follows. Assume for the sake of argument that $\calF$ is a finite dimensional manifold. The de Rham complex of $\Omega(\calF \times M)$ has a bigrading with vertical differential $\delta$ in the direction of $\calF$ and horizontal differential in the direction of $M$. A vector field $\xi$ on $\calF \times M$ is the lift of a vector field on $\calF$ if and only $[\iota_\xi, d] = 0$ and a lift of a vector field on $M$ if and only if $[\iota_\xi, \delta] = 0$. Props.~\ref{prop:StrictVert} and \ref{prop:StrictHor} show that this characterization is valid also in the variational bicomplex. In order to emphasize this geometric interpretation, we will use for the purpose of this paper the following terminology:

\begin{Definition}
\label{def:StrictlyVertHor}
A vector field $\xi$ on $J^\infty F$ will be called \textbf{strictly vertical} if $[\iota_\xi, d] = 0$ and \textbf{strictly horizontal} if $[\iota_\xi, \delta] = 0$.
\end{Definition}

The Lie derivatives of a strictly vertical vector field $\xi$ and of a strictly horizontal vector field $\hat{v}$ are given by
\begin{equation*}
  \Lie_\xi = [\iota_\xi, \delta]
  \,,\qquad
  \Lie_{\hat{v}} = [\iota_{\hat{v}}, d]
  \,.
\end{equation*}

\begin{Definition}
A strictly vertical vector field $\xi$ such that $\Lie_\xi L = d\beta$ for some $\beta \in \Omega^{0,n-1}(J^\infty F)$ will be called a \textbf{Noether symmetry} of the LFT.
\end{Definition}

\begin{Remark}
\label{rmk:NoetherTerminology}
Vector fields on the infinite jet bundle are sometimes called ``generalized vector fields'' and symmetries given by such vector fields ``generalized symmetries'' (e.g.~in \cite{DeligneFreed:1999}). However, an analysis of Noether's historic paper shows that this is Noether's original notion of symmetry, which was only to be rediscovered later \cite[Sec.~7.1]{Kosmann-Schwarzbach:Noether}.
\end{Remark}

Recall that a form $j \in \Omega^{0,n-1}(J^\infty F)$ is also called a \textbf{current}. If there is a strictly vertical vector field $\xi$ such that
\begin{equation*}
    dj = \iota_\xi EL \,,
\end{equation*}
then $j$ is called a \textbf{Noether current} and $(\xi,j)$ a \textbf{Noether pair} \cite[Def.~2.97]{DeligneFreed:1999}. Noether currents are conserved. Noether's first theorem states that if $\xi$ is a Noether symmetry, then
\begin{equation*}
  j := \beta - \iota_\xi \gamma
\end{equation*}
is a Noether current. The proof is a half-line calculation,
\begin{equation*}
  dj 
  = d\beta - d\iota_\xi \gamma
  = \iota_\xi \delta L + \iota_\xi d \gamma
  = \iota_\xi EL
  \,,
\end{equation*}
which highlights the advantage of working in the variational bicomplex.

\subsection{Manifest diffeomorphism symmetries}
\label{sec:Manifest}

In \cite[p.~169]{DeligneFreed:1999}, a \textbf{manifest symmetry} was defined to be a vector field $X \in \calX(J^\infty F)$ such that:
\begin{itemize}

\item[(i)] $X = \xi + \hat{v}$ is the sum of a strictly vertical vector field $\xi$ and a strictly horizontal vector field $\hat{v}$.

\item[(ii)] $\Lie_{\xi + \hat{v}} (L + \gamma) = 0$.

\end{itemize}
This suggests the following terminology:

\begin{Definition}
\label{def:ManifestDiffSym}
Let $(M,F,L)$ be a LFT with boundary form $\gamma$. An action
\begin{align*}
  \rho: \calX(M)
  &\longrightarrow \calX(J^\infty F)
  \\
  v &\longmapsto \rho(v) := \xi_v + \hat{v}
  \,.
\end{align*}
by manifest symmetries will be called a \textbf{manifest diffeomorphism symmetry}. 
\end{Definition}

\begin{Remark}
The Cartan lift $v \mapsto \hat{v}$ of vector fields on $M$ is a homomorphism of Lie algebras. Since strictly vertical and strictly horizontal vector fields commute, it follows that the map $v \mapsto \xi_v$ is a homomorphism of Lie algebras, too.
\end{Remark}

\begin{Remark}
If $F \to M$ is a natural bundle, i.e.~diffeomorphisms between open subsets of $M$ lift functorially to diffeomorphisms between local sections, then it follows from \cite{EpsteinThurston:1997} that we have an action of vector fields on $J^\infty F$. The diffeomorphism symmetries of LFTs often arise in this way.
\end{Remark}

\begin{Proposition}
\label{prop:ManifestMomMap}
Let $(M,F,L)$ be an LFT with boundary form $\gamma$. Then every manifest diffeomorphism symmetry $\rho: \calX(M) \to \calX(J^\infty F)$ has a homotopy momentum map 
\begin{equation*}
  \mu: \calX(M) \longrightarrow 
  L_\infty(J^\infty F, EL + \delta\gamma)
  \,.
\end{equation*}
given by 
\begin{equation*}
  \mu_k(v_1, \ldots, v_k) 
  := \iota_{\rho(v_1)} \cdots \iota_{\rho(v_k)} (L + \gamma)
  \,.  
\end{equation*}
\end{Proposition}
\begin{proof}
This is a special case of Prop.~\ref{prop:ExactMomentum}.
\end{proof}

The homotopy momentum map of a single vector field is split into a bidegree $(0,n-1)$ and a bidegree $(1,n-2)$ summand as
\begin{equation}
\label{eq:mu1}
\begin{split}
  \mu_1(v) 
  &= (\iota_{\xi_v} + \iota_{\hat{v}})(L + \gamma)
  = (\iota_{\hat{v}} L + \iota_{\xi_v} \gamma) 
  + \iota_{\hat{v}} \gamma
  \\
  &= - j_v + \iota_{\hat{v}} \gamma
  \,,
\end{split}
\end{equation}
where
\begin{equation}
\label{eq:NoetherCurrLFT}
  j_v = - \iota_{\hat{v}} L - \iota_{\xi_v} \gamma 
\end{equation}
is the Noether current of $\xi_v$. In general, the map $\mu_k$ splits into a $(0,n-k)$ and a $(1,n-k-1)$ component given by the two lines of the right hand side of the equation
\begin{equation*}
\begin{split}
  \mu_k(v_1, \ldots, v_k)
  &= - \sum_{i = 1}^k (-1)^{k-i} 
     (\iota_{\hat{v}_1} \cdots \widehat{\iota_{\hat{v}_i}}
     \cdots \iota_{\hat{v}_k} ) j_{v_i}
  +  (1-k)(\iota_{\hat{v}_1} \cdots \iota_{\hat{v}_k}) L  
  \\
  &{}\quad
  + (\iota_{\hat{v}_1} \cdots \iota_{\hat{v}_k}) \gamma
  \,.
\end{split}    
\end{equation*}
For example, we have
\begin{equation*}
 \mu_2(v, w)
 = (\iota_{\hat{v}} j_{w} - \iota_{\hat{w}} j_{v}
 + \iota_{\hat{v}}\iota_{\hat{w}} L)
 + \iota_{\hat{v}} \iota_{\hat{w}} \gamma
\end{equation*}
Using Eq.~\eqref{eq:2bracket}, we can write the $l_2$-bracket of the momenta as
\begin{equation}
\label{eq:TwoBracketSplit}
\begin{split}
  \{ \mu_1(v), \mu_1(w) \}
  &= \iota_{\xi_{w}} \iota_{\xi_{v}} \delta\gamma
  + ( \iota_{\hat{w}} \iota_{\xi_{v}} 
     -\iota_{\hat{v}} \iota_{\xi_{w}} ) EL
  \\
  &{}\quad
  + (\iota_{\hat{w}} \iota_{\xi_v} 
    -\iota_{\hat{v}} \iota_{\xi_w}) \delta\gamma
  + \iota_{\hat{w}} \iota_{\hat{v}} EL
  \\
  &{}\quad
  + \iota_{\hat{w}} \iota_{\hat{v}} \delta\gamma
  \,,
\end{split}    
\end{equation}
where the three lines of the right hand side are of bidegrees $(0,n-1)$, $(1,n-2)$, and $(2,n-3)$. The right hand side of Eq.~\eqref{eq:muBracket} is expressed in terms of the Noether current as
\begin{equation*}
\begin{split}
  \mu_1([v,w]) - \Dtot \mu_2(v,w)
  &= - j_{[v,w]} 
     + d( \iota_{\hat{v}} j_{w} - \iota_{\hat{w}} j_{v}
         +\iota_{\hat{v}}\iota_{\hat{w}} L)
  \\
  &{}\quad
  + \iota_{\widehat{[v,w]}} \gamma
  - \delta( \iota_{\hat{v}} j_{w} - \iota_{\hat{w}} j_{v}
           +\iota_{\hat{v}}\iota_{\hat{w}} L)
  - d \iota_{\hat{v}} \iota_{\hat{w}} \gamma
  \\
  &{}\quad
  + \iota_{\hat{w}} \iota_{\hat{v}} \delta\gamma
  \,.
\end{split}    
\end{equation*}

\begin{Remark}
\label{rmk:IntegralBracket}
If we integrate $\mu_1(a)$ over a closed codimension 1 submanifold $\Sigma \subset M$, we see from Eq.~\eqref{eq:mu1} that we obtain, up to a sign, the usual Noether charge $\int_\Sigma \mu_1(v) = - \int_\Sigma j_v$. This is no longer true for the brackets. The integral $\int_\Sigma \iota_{\xi_{w}} \iota_{\xi_{v}} \delta\gamma$ of the first summand on the right hand side of Eq.~\eqref{eq:TwoBracketSplit} is the usual bracket of charges. The integral of the second summand, however, is an additional contribution, which is not present in the multimomentum map of \cite[Sec.~4.1]{BridgesEtAl:2010}. The integrals of all other terms on the right hand side of Eq.~\eqref{eq:TwoBracketSplit} vanish for degree reasons.
\end{Remark}

\begin{Example}[Classical mechanics]
In classical mechanics spacetime is time $M=\bbR$ and the configuration bundle is trivial, $F = \bbR \times Q \to \bbR$, so that $\calF = C^\infty(\bbR, Q)$ is the space of smooth paths in $Q$. Let us consider the lagrangian of a particle of mass 1 in a potential $V$,
\begin{equation*}
  L =  \bigl( \tfrac{1}{2} \dot{q}^i \dot{q}^i - V(q) \bigr)dt
  \,.
\end{equation*}
Here $t, q^i, \dot{q}^i, \ddot{q}^i, \ldots$ are coordinates on the infinite jet bundle, given by
\begin{equation*}
  \dot{q}^i(j^\infty_0 x) = \frac{dx^i}{dt}\Bigr|_{t=0}
\end{equation*}
for a path $x: \bbR \to Q$. Using the relations $d \delta q^i = -\delta\dot{q}^i \wedge dt$, $dq^i = \dot{q}^i dt$, and $d \dot{q}^i = \ddot{q}^i dt$, we find that $\delta L = EL - d\gamma$ with
\begin{align*}
  EL &= - \Bigl( \ddot{q}^i
  + \frac{\partial V}{\partial q^i} \Bigr) \delta q^i \wedge dt
  \\
  \gamma &= \dot{q}^i \delta q^i
  \,.
\end{align*}
For the presymplectic form $\omega$ we obtain
\begin{equation*}
  \omega = 
  - \Bigl( \ddot{q}^i
  + \frac{\partial V}{\partial q^i} \Bigr) \delta q^i \wedge dt
  + \delta\dot{q}^i \wedge \delta q^i
  \,,
\end{equation*}
which is a form on $J^2 (\bbR \times Q)$.
The Cartan lift of the infinitesimal generator of time translation, i.e~of the coordinate vector field $\partial_t \equiv \frac{\partial}{\partial t} \in \calX(\bbR)$ is
\begin{equation*}
  \hat{\partial}_t
  = \frac{\partial}{\partial t} 
  + \dot{q}^i\frac{\partial}{\partial q^i}
  + \ddot{q}^i\frac{\partial}{\partial \dot{q}^i}
  + \ldots
\end{equation*}
The time translation $x(\tau) \mapsto x(\tau - t)$ of paths descends to the strictly vertical vector field
\begin{equation*}
  \xi_{\partial_t}
  =  
  - \dot{q}^i\frac{\partial}{\partial q^i}
  - \ddot{q}^i\frac{\partial}{\partial \dot{q}^i}
  - \ldots 
\end{equation*}
The fundamental vector field of the diagonal action of time translation on $J^\infty F$ is therefore given by
\begin{equation}
\label{eq:QM1}
  \rho(\partial_t) = \xi_{\partial_t} + \hat{\partial}_t
  = \frac{\partial}{\partial t}
  \,.
\end{equation}
This equation looks like a tautology, but the vector field $\frac{\partial}{\partial t}$ on the right hand side is not horizontal and must not be identified with the vector field in the time direction. Moreover, $\rho$ is not $C^\infty(M)$-linear.

Eq.~\eqref{eq:QM1} implies that $\Lie_{\rho(\partial_t)} (L + \gamma) = 0$, so that time translation is a manifest symmetry. The corresponding momentum map is given by
\begin{equation*}
\begin{split}
  \mu_1(\partial_t) = - j_{\partial_t}
  \,,
\end{split}    
\end{equation*}
since for degree reasons the term $\iota_{\hat{\partial}_t} \gamma$ vanishes. The Noether momentum
\begin{equation*}
  j_{\partial_t} 
  = - \iota_{\hat{\partial}_t} L - \iota_{\xi_{\partial_t}} \gamma
  = \tfrac{1}{2} \dot{q}^i \dot{q}^i + V(q)
\end{equation*}
is the energy.
\end{Example}

\section{The variational bicomplex of lorentzian metrics}
\label{sec:VarBicomplex}

We turn to general relativity. Here, the fields are lorentzian metrics on the spacetime manifold $M$. Vector fields on $M$ act on metrics by the Lie derivative. This action is local, so that it descends to the infinite jet bundle, inducing an action on the variational bicomplex. In order to study this action, we introduce in Def.~\ref{def:CovariantForms} the concept of covariant and contravariant families of forms in the variational bicomplex, which generalizes the concept of tensor fields. In Sec.~\ref{sec:CovDerForms} we generalize the notion of covariant derivative to such families of forms. In Sec.~\ref{sec:VolumeForm} we derive divergence formulas that express the horizontal differential of a form in terms of the covariant derivative and the metric volume form. While the computations are similar to those with tensor fields, there are also differences. For example, the metric volume form is invariant (Lem.~\ref{lem:VolInvariant}), rather than transforming as a density.

\subsection{The action of spacetime vector fields}

Assume that $M$ is a manifold of finite dimension $n$. The configuration bundle of general relativity is the bundle of fibre-wise lorentzian metrics on the tangent spaces of the spacetime manifold $M$, which we denote by $\Lor \to M$. We use the ``east coast'' sign convention in which the signature of the metric is $(-1,1, \ldots, 1)$. The diffeological space of lorentzian metrics on $M$ will be denoted by $\calLor$.

\begin{Remark}
In many papers on LFTs and the variational bicomplex one of the the following simplifying assumptions about the configuration bundle $F \to M$ is made: $F$ is a vector bundle; the fibres of $F$ are connected; the space of sections $\calF = \Gamma(M,F)$ is non-empty; the jet evaluations $j^k: \calF \times M \to J^k F$ are surjective. All these assumptions generally fail for the bundle of lorentzian metrics.
\end{Remark}

The configuration bundle is natural, which means that local diffeomorphisms on $M$ lift functorially to the sheaf of sections. In particular, we have a left action of the diffeomorphism group $\Diff(M)$ on the space of fields $\calLor$ by pushforward. Infinitesimally, we have a left action of the Lie algebra of vector fields,
\begin{equation*}
\begin{aligned}
  \Xi: \calX(M) &\longrightarrow \calX(\calLor)
  \\
  v &\longmapsto (\Xi_v: \eta \mapsto -\Lie_v \eta)
  \,,
\end{aligned}
\end{equation*}
where the symmetric 2-form $-\Lie_v \eta$ represents a tangent vector in $T_\eta \calLor$. This action is local, so that it descends to an action of $\calX(M)$ on $J^\infty \Lor$ by strictly vertical vector fields,
\begin{equation*}
\begin{aligned}
  \xi: \calX(M) &\longrightarrow \calX(J^\infty \Lor)
  \\
  v &\longmapsto \xi_v
  \,.
\end{aligned}
\end{equation*}
Together with the Cartan lift of the vector field in $\calX(M)$, we obtain a homomorphism of Lie algebras
\begin{equation}
\label{eq:DiffActionGR}
\begin{aligned}
  \rho: \calX(M)
  &\longrightarrow \calX(J^\infty \Lor)
  \\
  v &\longmapsto \rho(v) := \xi_v + \hat{v}
  \,.
\end{aligned}
\end{equation}
Our ultimate goal is to show that $\rho$ is a manifest symmetry of general relativity for a natural choice of boundary form. In this section we will gather the necessary tools.

\subsection{Jet coordinates}

Let $(x^1, \ldots, x^n)$ be a system of local spacetime coordinates on an open subset $U \subset M$. The coordinate vector fields will be denoted by $\partial_a = \frac{\partial}{\partial x^a}$, the coordinate 1-forms by $dx^a$. A lorentzian metric $\eta \in \calLor$ is written in local coordinates as $\eta = \tfrac{1}{2} \eta_{ab} dx^a \wedge dx^b$, where $\eta_{ab} = \iota_{\partial_b} \iota_{\partial_a} \eta \in C^\infty(M)$ are the matrix components of the metric. (Recall that we use the Einstein summation convention throughout the paper.)

The local coordinates on $M$ induce local jet coordinates given by
\begin{equation*}
\begin{aligned}
  g_{ab,c_1 \cdots c_k}: 
  J^\infty \Lor 
  &\longrightarrow \bbR
  \\
  j^\infty_x \eta 
  &\longmapsto \frac{\partial^k \eta_{ab}}{
  \partial x^{c_1}  \cdots \partial x^{c_k} } \Bigr|_x 
  \,.  
\end{aligned}
\end{equation*}
Since the partial derivatives commute, $g_{ab,c_1 \cdots c_k}$ is invariant under permutations of the indices $c_1, \ldots, c_k$. To avoid overcounting in summation formulas it is convenient to use the multi-index notation of multi-variable analysis: A multi-index is a tuple $C = (C_1, \ldots, C_n)$ of natural numbers $C_k \geq 0$. The number $|C| = C_1 + \ldots + C_n$ is called the length of the index.  The concatenation of a multi-index with a single index is given by
\begin{equation*}
  Cd = (C_1, \ldots, C_d +1, \ldots, C_n)
  \,.
\end{equation*}
The jet coordinate function labeled by a multi-index is given by
\begin{equation*}
  g_{ab,C}(j^\infty_x \eta) 
  = \frac{\partial^{|C|} \eta_{ab}}{
  (\partial x^1)^{C_1}  \cdots (\partial x^n)^{C_n} } \Bigr|_x 
  \,.
\end{equation*}
The collection of functions $\{ x^a, g_{ab, C} \}$ for $1 \leq a \leq b \leq n$ and $C \in \bbN_0^n$ is a system of local coordinates on $J^\infty \Lor$.

\begin{Remark}
\label{rmk:PhysicsJetNotation}
In the physics literature, the same notation is usually used for both the jet coordinates and their evaluation on a field, which can be confusing. For example, if $M$ is non-compact, every $n$-form is exact, in particular the integrand $L(\eta)$ of the action. So for the step ``discarding exact terms'' during the derivation of the Euler-Lagrange equation to be meaningful, we must view the integrand as an element $L \in \Omega^{0,n}(J^\infty \Lor)$, i.e.~as an expression of the jet coordinates like $g_{ab,c}$ and not of the derivatives $\frac{\partial \eta_{ab}}{\partial x^c}$ of a particular metric $\eta$.
\end{Remark}

The variational bicomplex is generated as bigraded algebra by the coordinate functions, the vertical coordinate 1-forms $\delta g_{ab,C}$ in degree $(1,0)$, and the horizontal coordinate 1-forms $d x^a$ in degree $(0,1)$. A $(p,q)$-form is given in local coordinates by
\begin{equation*}
  \omega = \omega^{a_1, b_1, \ldots, a_p, b_p, C_1,\ldots, C_p}_{e_1, \ldots, e_q} 
  \delta g_{a_1 b_1, C_1} \wedge \ldots \wedge 
  \delta g_{a_p b_p, C_p}
  \wedge dx^{e_1} \wedge \ldots \wedge dx^{e_q}
  \,,
\end{equation*}
where the coefficients are functions on $J^\infty \Lor$. The other differentials of the jet coordinates are given by \cite[p.~18]{Anderson:1989}
\begin{equation*}
\begin{aligned}
  \delta x^a &= 0
  \\
  d g_{ab,C} 
  &= g_{ab,Ce}\, dx^e
  \,.
\end{aligned}
\end{equation*}
It follows that the differentials of the coordinate 1-forms are given by $d dx^a = 0$, $\delta \delta g_{ab,C} = 0$, $\delta dx^a = 0$, and \begin{equation*}
   d\delta g_{ab,C} = 
   - \delta g_{ab,Ce} \wedge dx^e
   \,.
\end{equation*}

Dually, the $C^\infty(J^\infty \Lor)$-module of vertical vector fields is spanned by the coordinate vector fields $\frac{\partial}{\partial g_{ab,C}}$, which satisfy
\begin{equation*}
\begin{aligned}
  \iota_{\frac{\partial}{\partial g_{ab,C}}} \delta g_{a'b',C'}
  &= \delta^a_{a'} \delta^b_{b'} \delta^C_{C'}
  \\
  \iota_{\frac{\partial}{\partial g_{ab,C}}} dx^e
  &= 0
  \,.
\end{aligned}
\end{equation*}
The module of horizontal vector fields, called the Cartan distribution, is spanned by the vector fields
\begin{equation*}
  \hat{\partial}_a 
  = \frac{\partial}{\partial x^a}
  + \sum_{|D| = 0}^\infty g_{bc,Da}
  \frac{\partial}{\partial g_{bc,D}}
  \,,
\end{equation*}
which satisfy
\begin{equation*}
\begin{aligned}
  \iota_{\hat{\partial}_a} \delta g_{bc,D} 
  &= 0
  \\
  \iota_{\hat{\partial}_a} dx^{a'} 
  &= \delta^{a'}_a 
  \,.  
\end{aligned}
\end{equation*}

The Cartan distribution can be viewed as an Ehresmann connection on the bundle $J^\infty \Lor \to M$. The horizontal lift of a vector field $v = v^a(x) \frac{\partial}{\partial x^a}$ on $M$ to $J^\infty \Lor$ is given by
\begin{equation*}
  \hat{v} = v^a(x) \hat{\partial}_a
  \,.
\end{equation*}
The vertical and horizontal differentials of a function $f \in C^\infty(J^\infty \Lor)$ are given by \cite[pp.~18-19]{Anderson:1989}
\begin{equation*}
\begin{aligned}
  \delta f &= \sum_{|C| = 0}^\infty
  \frac{\partial f}{\partial g_{ab,C}}  \delta g_{ab,C}
  \\
  d f &= (\hat{\partial}_a f) dx^a
  \,.
\end{aligned}    
\end{equation*}
The horizontal differential of a form $\omega \in \Omega^{p,q}(J^\infty \Lor)$ is given in local coordinates by
\begin{equation}
\label{eq:dGeneralForm}
  d\omega = (-1)^{p+q} (\Lie_{\hat{\partial}_a} \omega) \wedge dx^a 
  \,.
\end{equation}

A vector field is strictly horizontal if and only if it is the horizontal lift $\hat{v}$ of a vector field $v$ on $M$ by the Cartan connection. A vector field $\xi$ is strictly vertical if and only if it is the infinite prolongation of an evolutionary ``vector field'', i.e.~of a map $\xi_0: J^\infty \Lor \to V\Lor$ of bundles over $\Lor$, where $V\Lor \subset T\Lor$ is the vertical tangent bundle. In local coordinates it is of the form
\begin{equation}
\label{eq:ProlongEvol}
  \xi = \sum_{|C|=0}^\infty (\hat{\partial}_C \xi_{ab})
  \frac{\partial}{\partial g_{ab,C}}
  \,,
\end{equation}
where the $\xi_{ab}$ are functions on $J^\infty \Lor$ and where $\hat{\partial}_C = (\hat{\partial}_1)^{C_1} \cdots (\hat{\partial}_n)^{C_n}$ is the multi-index notation for the iterated application of the horizontal lifts of the coordinate vector fields.

\subsection{Action of spacetime vector fields on infinite jets}

The action of a vector field $v \in \calX(M)$ on a lorentzian metric $\eta \in \calLor$ by the negative Lie derivative, $\eta \mapsto -\Lie_v \eta$, is given in local coordinates by
\begin{equation*}
\begin{split}
  \eta_{ab} dx^a dx^b
  \longmapsto
  - \Bigl( v^c \frac{\partial \eta_{ab}}{\partial x^c}
  + \frac{\partial v^{a'}}{\partial x^a} \eta_{a'b}
  + \frac{\partial v^{b'}}{\partial x^b} \eta_{ab'}  
  \Bigr) dx^a dx^b
  \,.
\end{split}
\end{equation*}
We can view this as transformation of the coordinate functions
\begin{equation}
\label{eq:xivEvol}
  g_{ab}
  \longmapsto
  - \Bigl( v^c g_{ab,c} 
  + \frac{\partial v^{a'}}{\partial x^a} g_{a'b}
  + \frac{\partial v^{b'}}{\partial x^b} g_{ab'}
  \Bigr)
  =: \xi_{ab}
  \,,
\end{equation}
which are the components of the evolutionary ``vector field'' $\xi_{ab} \frac{\partial}{\partial g_{ab}}$. Its infinite prolongation is the strictly vertical vector field
\begin{equation*}
  \xi_v 
  =  \sum_{|C| = 0}^\infty (\hat{\partial}_C\xi_{ab})
     \frac{\partial}{\partial g_{ab,C}}
  \,,
\end{equation*}
which defines the action~\eqref{eq:DiffActionGR} of vector fields on the infinite jet bundle.

\subsection{Covariant and contravariant families of forms}

The Lie derivative of a coordinate function with respect to a strictly horizontal vector field is given by
\begin{equation*}
  \Lie_{\hat{v}} g_{ab,C}
  = \iota_{v^e\hat{\partial}_e} d g_{ab,C} \\
  = v^e g_{ab,Ce} \,.    
\end{equation*}
In particular, we have
\begin{equation*}
  \Lie_{\hat{\partial}_e} g_{ab}
  = g_{ab,e}
  \,.    
\end{equation*}
Note that this is the Lie derivative of a single function $g_{ab} \in C^\infty(J^\infty \Lor)$ and must not be confused with the Lie derivative of a metric 2-form on $M$. The formula~\eqref{eq:xivEvol} for the 0-jet component $\xi_{ab}$ of $\xi_v$ can now be written as
\begin{equation*}
  \Lie_{\xi_v} g_{ab}
  = - \Lie_{\hat{v}} g_{ab} 
  - \frac{\partial v^{a'}}{\partial x^a} g_{a'b}
  - \frac{\partial v^{b'}}{\partial x^b} g_{ab'}
  \,.
\end{equation*}
This can be expressed in terms of the diagonal action $\rho$ as
\begin{equation}
\label{eq:Lieg}
  \Lie_{\rho(v)} g_{ab}
  = 
  - \frac{\partial v^{a'}}{\partial x^a} g_{a'b}
  - \frac{\partial v^{b'}}{\partial x^b} g_{ab'}
  \,.
\end{equation}
Since $\delta$ commutes with both $\Lie_{\xi_v}$ and $\Lie_{\hat{v}}$, it commutes with $\Lie_{\rho(v)}$. This implies that
\begin{equation}
\label{eq:Liedeltag}
  \Lie_{\rho(v)} \delta g_{ab}
  =
  - \frac{\partial v^{a'}}{\partial x^a} \delta g_{a'b}
  - \frac{\partial v^{b'}}{\partial x^b} \delta g_{ab'}
  \,.
\end{equation}
Using $g^{ab}g_{bc} = \delta^a_c$, we get
\begin{equation}
\label{eq:Lieginv}
  \Lie_{\rho(v)} g^{ab}
  =  
    \frac{\partial v^a}{\partial x^{a'}} g^{a'b}
  + \frac{\partial v^b}{\partial x^{b'}} g^{ab'}
  \,.
\end{equation}
These calculations suggest the following definition.

\begin{Definition}
\label{def:CovariantForms}
A family of forms $\chi_{a_1 \cdots a_p}^{b_1 \cdots b_q} \in \Omega(J^\infty \Lor)$, $1 \leq a_1, \ldots, b_q \leq n$ is called \textbf{covariant} in $a_1, \ldots, a_p$ and \textbf{contravariant} in $b_1, \ldots, b_q$ if
\begin{equation*}
  \Lie_{\rho(v)} \chi_{a_1 \cdots a_p}^{b_1 \cdots b_q}
  = 
  - \sum_{i=1}^p
  \frac{\partial v^{a'_i}}{\partial x^{a_i}}\, 
  \chi_{a_1 \cdots a'_i \cdots  a_p}^{b_1 \cdots b_q}
  + \sum_{i=1}^q
  \frac{\partial v^{b_i}}{\partial x^{b'_i}}\, 
  \chi_{a_1 \cdots a_p}^{b_1 \cdots  b'_i \cdots  b_q}
  \,.  
\end{equation*}
A form $\chi \in \Omega(J^\infty \Lor)$ is called \textbf{invariant} if $\Lie_{\rho(v)} \chi = 0$.
\end{Definition}

Def.~\ref{def:CovariantForms} generalizes the notion of covariant and contravariant tensors to families of forms in $\Omega(J^\infty \Lor)$. In this terminology Eqs.~\eqref{eq:Lieg}, \eqref{eq:Liedeltag}, and \eqref{eq:Lieginv} show that the indices of $g_{ab}$ and $\delta g_{ab}$ are covariant, while those of $g^{ab}$ and $\delta g^{ab}$ are contravariant. Covariant and contravariant families behave in many ways as tensors.

\begin{Lemma}
\label{lem:WedgeCovariant}
Let $\chi_a$ be a covariant and $\psi^b$ a contravariant family of forms. Then the family $\chi_a \wedge \psi^b$ is covariant in $a$ and contravariant in $b$.
\end{Lemma}
\begin{proof}
This follows immediately from the fact that $\Lie_{\rho(v)}$ is a degree 0 derivation of the algebra $\Omega(J^\infty \Lor)$.
\end{proof}

\begin{Lemma}
\label{lem:ContractionInvariant}
Let $\chi_a^b$ be a family of forms that is covariant in $a$ and contravariant in $b$, then the contracted form $\chi_a^a$ (summation over $a$) is invariant.
\end{Lemma}

The last two lemmas generalize in an obvious way to families with several indices. An immediate consequence of Lem.~\ref{lem:WedgeCovariant} and Lem.~\ref{lem:ContractionInvariant} is that we can raise and lower indices with the metric coordinate functions in the usual way: If $\chi_a$ is covariant, then $\chi^a = g^{aa'}\chi_{a'}$ is contravariant. If $\chi^a$ is contravariant, then $\chi_a = g_{aa'}\chi^{a'}$ is covariant.

\begin{Lemma}
\label{lem:InnnerDerConvariant}
If the family $\chi_b \in \Omega(J^\infty \Lor)$ is covariant, then the family $\iota_{\hat{\partial}_a} \chi_b$ is covariant in $a$ and $b$. 
\end{Lemma}
\begin{proof}
Let $\psi_{ab} = \iota_{\hat{\partial}_a} \chi_b$
We have
\begin{equation*}
\begin{split}
  \Lie_{\rho(v)} \psi_{ab}
  &=
  \Lie_{\xi_v + \hat{v}} (\iota_{\hat{\partial}_a} \chi_b)
  \\
  &= \bigl(\iota_{\hat{\partial}_a} \Lie_{\xi_v}
  + \iota_{\hat{\partial}_a}\Lie_{\hat{v}}
  + \iota_{[\hat{v}, \hat{\partial}_a]} \bigr) \chi_b
  \\
  &= \iota_{\hat{\partial}_a} \Lie_{\rho(v)} \chi_b
  - \frac{\partial v^{a'}}{\partial x^a} (\iota_{\hat{\partial}_{a'}}\chi_b)
  \\
  &= 
  - \frac{\partial v^{b'}}{\partial x^b} \psi_{ab'}
  - \frac{\partial v^{a'}}{\partial x^a} \psi_{a'b}
  \,,
\end{split}
\end{equation*}
which shows that $\psi_{ab}$ is covariant in $a$ and $b$.
\end{proof}

\begin{Lemma}
If the family $\chi_a \in \Omega(J^\infty \Lor)$ is covariant, then the family $\delta\chi_b$ is covariant. 
\end{Lemma}
\begin{proof}
We have
\begin{equation*}
\begin{split}
  \Lie_{\rho(v)} \delta\chi_a
  &=
  \delta \Lie_{\rho(v)} \chi_a
  \\
  &=
  \delta \Bigl( 
  - \frac{\partial v^{a'}}{\partial x^a} \chi_{a'} \Bigr)
  \\
  &=
  - \frac{\partial v^{a'}}{\partial x^a} \delta\chi_{a'}
  \,,
\end{split}
\end{equation*}
which shows that $\chi_a$ is covariant.
\end{proof}

The last lemma generalizes in an obvious way to families of forms with covariant and contravariant indices. The analogous statement for the horizontal differential works only for invariant forms:

\begin{Lemma}
If $\chi \in \Omega(J^\infty \Lor)$ is invariant, then $d\chi$ is invariant. 
\end{Lemma}
\begin{proof}
The differential $d$ commutes with $\Lie_{\rho(v)}$, so that $\Lie_{\rho(v)} d\chi = d \Lie_{\rho(v)} \chi = 0$.
\end{proof}

\begin{Lemma}
\label{lem:LieDerConvariant}
If the form $\chi \in \Omega(J^\infty \Lor)$ is invariant, then the family $\Lie_{\hat{\partial}_a} \chi$ is covariant.
\end{Lemma}
\begin{proof}
We have
\begin{equation*}
\begin{split}
  \Lie_{\rho(v)} (\Lie_{\hat{\partial}_a} \chi)
  &=
  \Lie_{\xi_v + \hat{v}} (\Lie_{\hat{\partial}_a} \chi)
  \\
  &= \bigl(\Lie_{\hat{\partial}_a} \Lie_{\xi_v}
  + \Lie_{\hat{\partial}_a}\Lie_{\hat{v}}
  + \Lie_{[\hat{v}, \hat{\partial}_a]} \bigr) \chi
  \\
  &= \Lie_{\hat{\partial}_a} \Lie_{\xi_v + \hat{v}} \chi
  - \frac{\partial v^{a'}}{\partial x^a} (\Lie_{\hat{\partial}_{a'}}\chi)
  \\
  &= - \frac{\partial v^{a'}}{\partial x^a} (\Lie_{\hat{\partial}_{a'}}\chi)
  \,,
\end{split}
\end{equation*}
which shows that $\Lie_{\hat{\partial}_a} \chi$ is a covariant family.
\end{proof}

Lemma~\ref{lem:LieDerConvariant} holds only for an invariant form $\chi$. If $\chi_b$ is a covariant family, then $\Lie_{\hat{\partial}_a} \chi_b$ is \emph{not} covariant. In order to obtain a covariant family by differentiation we have to generalize the concept of covariant derivative to families of forms in the variational bicomplex.

\subsection{Covariant derivative of families of forms}
\label{sec:CovDerForms}

In the cohomological approach to general relativity, we have to interpret the connection coefficients, the covariant derivative, the curvature, the volume form, etc.~as expressions in the variational bicomplex. The connection coefficients of the Levi-Civita connection have to be viewed as functions on $J^\infty \Lor$ that are given in local coordinates by the expression
\begin{equation}
\label{eq:ConnectionCoeffs}
  \Gamma^a_{bc}
  = \tfrac{1}{2} g^{ad}(g_{db,c} + g_{dc,b} - g_{bc,d})
  \,.
\end{equation}
The covariant derivative has to be defined in the variational bicomplex as follows. For a family of forms $\chi_{a_1 \cdots a_p}^{b_1 \cdots b_q}$ that is covariant in the lower indices and contravariant in the upper indices we define
\begin{equation*}
  \nabla_c \chi_{a_1 \cdots a_p}^{b_1 \cdots b_q}
  = \Lie_{\hat{\partial}_c} \chi_{a_1 \cdots a_p}^{b_1 \cdots b_q}
  - \sum_{i=1}^p
  \Gamma_{c a_i}^{a'_i}\, 
  \chi_{a_1 \cdots a'_i \cdots  a_p}^{b_1 \cdots b_q}
  + \sum_{i=1}^q
  \Gamma_{c b'_i}^{b_i}\, 
  \chi_{a_1 \cdots a_p}^{b_1 \cdots  b'_i \cdots  b_q}
  \,.  
\end{equation*}
Using this definition, we can check by the usual calculation that the connection coefficients~\eqref{eq:ConnectionCoeffs} of the Levi-Civita connection is the unique family of functions symmetric in $b$ and $c$, such that $\nabla_c g_{ab} = 0$. The Riemann curvature tensor is given by $\Riem_{abc}{}^d\, \chi_d = (\nabla_a \nabla_b - \nabla_b \nabla_a) \chi_c$, which has now to be viewed as a family of functions on $J^\infty \Lor$.

\begin{Lemma}
\label{lem:NablaCovariant}
Let $\chi_b$ be a covariant family of vertical forms. Then the family $\nabla_a \chi_b$ is covariant in $a$ and $b$.
\end{Lemma}
\begin{proof}
We have to compute the Lie derivative of $\nabla_a \chi_b = \Lie_{\hat{\partial}_a} \chi_b - \Gamma^c_{ab} \chi_c$ with respect to $\rho(v) = \xi_v + \hat{v}$. For the first summand we get
\begin{equation*}
\begin{split}
  \Lie_{\rho(v)}( \Lie_{\hat{\partial}_a} \chi_b)
  &=
  \Lie_{\hat{\partial}_a} (\Lie_{\xi_v + \hat{v}} \chi_b)
  + \Lie_{[\hat{v}, \hat{\partial}_a]} \chi_b
  \\
  &=
  \Lie_{\hat{\partial}_a} \Bigl(
  - \frac{\partial v^{b'}}{\partial x^b} \chi_{b'} 
  \Bigr)
  + \iota_{[\hat{v}, \hat{\partial}_a]} d\chi_b
  \\
  &=
  - \frac{\partial^2 v^{b'}}{\partial x^a \partial x^b} \chi_{b'}
  - \frac{\partial v^{b'}}{\partial x^b} (\Lie_{\hat{\partial}_a} \chi_{b'})
  - \frac{\partial v^{a'}}{\partial x^a} \iota_{\hat{\partial}_{a'}} d\chi_b
  \\
  &=
  - \frac{\partial v^{b'}}{\partial x^b} (\Lie_{\hat{\partial}_a} \chi_{b'})
  - \frac{\partial v^{a'}}{\partial x^a} (\Lie_{\hat{\partial}_{a'}} \chi_b)
  - \frac{\partial^2 v^{c}}{\partial x^a \partial x^b} \chi_{c}
  \,.
\end{split}    
\end{equation*}
For the second summand we must compute the Lie derivative of the connection coefficients. For this we need the following formula.
\begin{equation*}
\begin{split}
  \Lie_{\xi_v} g_{ab,c}
  &= \Lie_{\xi_v} \Lie_{\hat{\partial}_c} g_{ab}
  \\
  &= \Lie_{\hat{\partial}_c} \Lie_{\xi_v} g_{ab}
  \\
  &= - \Lie_{\hat{\partial}_c}
  \Bigl(
  \Lie_{\hat{v}} g_{ab} 
  + \frac{\partial v^{a'}}{\partial x^a} g_{a'b}
  + \frac{\partial v^{b'}}{\partial x^b} g_{ab'}
  \Bigr)
  \\
  &= - (\Lie_{[\hat{\partial}_c, \hat{v}]} 
  + \Lie_{\hat{v}} \Lie_{\hat{\partial}_c}) g_{ab}
  \\
  &\quad{}
  - \frac{\partial^2 v^{a'}}{\partial x^c \partial x^a} g_{a'b}
  - \frac{\partial v^{a'}}{\partial x^a} \delta g_{a'b, c}
  - \frac{\partial^2 v^{b'}}{\partial x^c \partial x^b} g_{ab'}
  - \frac{\partial v^{b'}}{\partial x^b} g_{ab', c}
  \\
  &= - \Lie_{\hat{v}} g_{ab,c}
  - \frac{\partial v^{a'}}{\partial x^a} g_{a'b, c}
  - \frac{\partial v^{b'}}{\partial x^b} g_{ab', c}
  - \frac{\partial v^{c'}}{\partial x^c} g_{ab, c'}
  \\
  &\quad{}
  - \frac{\partial^2 v^{a'}}{\partial x^c \partial x^a} g_{a'b}
  - \frac{\partial^2 v^{b'}}{\partial x^c \partial x^b} g_{ab'}
\end{split}
\end{equation*}
With this relation, we can compute the action of vector fields on the connection coefficients, which yields
\begin{equation*}
  \Lie_{\rho(v)} \Gamma^c_{ab}
  =
  \frac{\partial v^c}{\partial x^{c'}} \Gamma^{c'}_{ab}
  - \frac{\partial v^{a'}}{\partial x^{a}} \Gamma^{c}_{a'b}
  - \frac{\partial v^{b'}}{\partial x^{b}} \Gamma^{c}_{ab'}
  - \frac{\partial^2 v^c}{\partial x^a \partial x^b}
  \,.
\end{equation*}
Putting everything together, we obtain
\begin{equation*}
\begin{split}
  \Lie_{\rho(v)}( \nabla_a \chi_b)
  &=
  \Lie_{\rho(v)} \Lie_{\hat{\partial}_a} \chi_b
  - (\Lie_{\rho(v)} \Gamma^c_{ab}) \chi_c
  - \Gamma^c_{ab}(\Lie_{\rho(v)} \chi_c)
  \\
  &= \frac{\partial v^{a'}}{\partial x^a}(\nabla_{a'} \chi_b)
  + \frac{\partial v^{b'}}{\partial x^a}(\nabla_{a} \chi_{b'})
  \,,
\end{split}    
\end{equation*}
where the terms containing the second order derivatives of $v^a$ cancel. This finishes the proof.
\end{proof}

\subsection{Divergence formulas}
\label{sec:VolumeForm}

In the variational bicomplex, the metric volume form is the $(0,n)$-form on $J^\infty \Lor$ defined as 
\begin{equation}
  \Volg = \sqrt{- \det g}\,\, dx^1 \wedge \ldots \wedge dx^n
  \,.
\end{equation}
We recall that we have adopted the ``east coast'' sign convention for Lorentz metrics with 1 negative and $n-1$ positive signs, so that $\det g$ is negative. The partial derivative of the square root of the determinant with respect to the 0-jet coordinates is given by
\begin{equation*}
  \frac{\partial}{\partial g_{ab}} \sqrt{- \det g}
  = \tfrac{1}{2} g^{ab} \sqrt{- \det g}
  \,.
\end{equation*}
The partial derivatives with respect to $x^a$ and all higher jet coordinates $g_{ab,C}$ vanish. It follows that the vertical and the horizontal differentials are given by 
\begin{align*}
  \delta \sqrt{- \det g}
  &= \tfrac{1}{2} g^{ab} \delta g_{ab} \sqrt{- \det g}
  \\
  d \sqrt{- \det g}
  &= \tfrac{1}{2} g^{ab} g_{ab,c} \sqrt{- \det g} \,\, dx^c
  \,.
\end{align*}
For the vertical differential of the volume form we obtain
\begin{equation}
\label{eq:deltaVol}
  \delta \Volg = \tfrac{1}{2} g^{ab}\delta g_{ab} \Volg 
  \,.
\end{equation}

Although $\Volg$ is not a volume form on $J^\infty \Lor$, every $(0,n)$-form $\tau$ can be written as 
\begin{equation*}
  \tau = f dx^1 \wedge \ldots \wedge dx^n
  = \frac{f}{\sqrt{-\det g}} \Volg
  \,,
\end{equation*}
for a unique function $f \in C^\infty(J^\infty \Lor)$. Therefore, we can define the divergence of a vector field $X \in \calX(J^\infty \Lor)$ by the relation
\begin{equation*}
  \Lie_X \Volg = (\Div\, X) \Volg
  \,.
\end{equation*}
For a strictly horizontal vector field $\hat{v}$ we have
\begin{equation}
\label{eq:StrictHorDivergence}
\begin{split}
  \Lie_{\hat{v}} \Volg
  &= (\Lie_{\hat{v}} \sqrt{-\det g}) \,\, 
  dx^1 \wedge \ldots \wedge dx^n
  + \sqrt{-\det g} \,\, 
  \Lie_{\hat{v}}(dx^1 \wedge \ldots \wedge dx^n)
  \\
  &= \Bigl( \tfrac{1}{2} g^{ab} g_{ab,c} v^c 
  + \frac{\partial v^c}{\partial x^c} \Bigr)
  \Volg
  = \Bigl( \Gamma^{a}_{ac} v^c
  + \frac{\partial v^c}{\partial x^c} \Bigr)
  \Volg
  \\ 
  &= (\nabla_a v^a) \Volg
  \,.
\end{split}
\end{equation}
We conclude that
\begin{equation*}
  \Div\, \hat{v} = \nabla_a v^a
  \,.
\end{equation*}
While this looks like the usual expression, we point out that the divergence $\Div\, \hat{v}$ is now a function on $J^1 \Lor$. 

\begin{Lemma}
\label{lem:VolInvariant}
The metric volume form is invariant.
\end{Lemma}
\begin{proof}
For the Lie derivative of the volume form with respect to the vertical vector field we obtain
\begin{equation*}
\begin{split}
  \Lie_{\xi_v} \Volg
  &= \iota_{\xi_v} \delta \Volg
  \\
  &= -\tfrac{1}{2} g^{ab}
  \Bigl(\Lie_{\hat{v}} g_{ab} 
  + \frac{\partial v^{a'}}{\partial x^a} g_{a'b}
  + \frac{\partial v^{b'}}{\partial x^b} g_{ab'} \Bigr) \Volg
  \\
  &=  
  - \Bigl(v^c \tfrac{1}{2} g^{ab} g_{ab,c} 
  + \frac{\partial v^{c}}{\partial x^c} \Bigr) \Volg
  \\
  &= - \Lie_{\hat{v}} \Volg
  \,,
\end{split}
\end{equation*}
where in the last step we have used Eq.~\eqref{eq:StrictHorDivergence}. We conclude that $\Lie_{\rho(v)} \Volg = 0$ for all $v \in \calX(M)$.
\end{proof}

\begin{Remark}
Lem.~\ref{lem:VolInvariant} can be stated by saying that $\xi_v + \hat{v}$ is divergence free.
\end{Remark}

From the formula for the divergence of a vector field we deduce
\begin{equation*}
  (\nabla_a v^a) \Volg
  = d\bigl( v^a \iota_{\hat{\partial}_a} \Volg)
  \,.
\end{equation*}
This formula generalizes to higher vertical forms, as we will show next.

\begin{Proposition}
\label{prop:DivFormula1}
Let $\chi^a$ be a family of $(p,0)$-forms on $J^\infty \Lor$. Then \begin{equation}
\label{eq:dDiv1}
  \nabla_a \chi^a \wedge \Volg 
  = (-1)^p d (\chi^a \wedge \iota_{\hat{\partial}_a} \Volg)
  \,.
\end{equation}
\end{Proposition}
\begin{proof}
Consider the $(p,n-1)$-form
\begin{equation*}
  \chi = (-1)^p\chi^a \wedge \iota_{\hat{\partial}_a} \Volg
  \,,
\end{equation*}
where the $\chi^a$ are $(p,0)$-forms. The horizontal differential of $\chi$ is given by
\begin{equation*}
\begin{split}
  d\chi 
  &= (-1)^{p+n-1} (\Lie_{\hat{\partial}_c}\chi) \wedge dx^c
  \\
  &= (-1)^{n-1} \Lie_{\hat{\partial}_c} 
  (\chi^a \wedge \iota_{\hat{\partial}_a} \Volg) \wedge dx^c
  \\
  &= (\Lie_{\hat{\partial}_c} \chi^a) \wedge 
  (-1)^{n-1}( \iota_{\hat{\partial}_a} \Volg) \wedge dx^c
  + \chi^a \wedge (-1)^{n-1} (\Lie_{\hat{\partial}_c}
  \iota_{\hat{\partial}_a} \Volg) \wedge dx^c
  \\
  &= (\Lie_{\hat{\partial}_a} \chi^a) \wedge \Volg
  + \chi^a \wedge (-1)^{n-1}  
  (\iota_{\hat{\partial}_a} \Lie_{\hat{\partial}_c} \Volg) 
  \wedge dx^c
  \\
  &= (\Lie_{\hat{\partial}_a} \chi^a) \wedge \Volg 
  + \chi^a \wedge (-1)^{n-1}\iota_{\hat{\partial}_a}
  (\Gamma^b_{bc} \Volg) \wedge dx^c
  \\
  &= (\Lie_{\hat{\partial}_a} \chi^a 
  + \Gamma^b_{ba}\chi^a) \wedge \Volg 
  \\
  &= (\nabla_a \chi^a) \wedge \Volg
  \,,
\end{split}    
\end{equation*}
where we have used Eq.~\eqref{eq:dGeneralForm}, the Leibniz rule, the relation
\begin{equation*}
  (\iota_{\hat{\partial}_a}\Volg) \wedge dx^c
  = (-1)^{n-1}\delta^c_a \Volg
  \,,
\end{equation*}
and Eq.~\eqref{eq:StrictHorDivergence}.
\end{proof}

For later use, we generalize the formula~\eqref{eq:dDiv1} further to families of $(p,1)$-forms.

\begin{Proposition}
\label{prop:DivFormula2}
Let $\chi^{ab}$ be a family of $(p,1)$-forms on $J^\infty \Lor$ such that $\chi^{ab} = - \chi^{ba}$. Then 
\begin{equation}
\label{eq:dDiv2}
  \nabla_a \chi^{ab} \wedge \iota_{\hat{\partial}_b} \Volg 
  = (-1)^p d (\tfrac{1}{2}\chi^{ab} \wedge 
  \iota_{\hat{\partial}_a}\iota_{\hat{\partial}_b} \Volg)
  \,.
\end{equation}
\end{Proposition}
\begin{proof}
Consider the $(p, n-2)$-form
\begin{equation*}
  \chi = \tfrac{1}{2}(-1)^p\chi^{ab} \wedge 
  \iota_{\hat{\partial}_a}\iota_{\hat{\partial}_b} \Volg
  \,.
\end{equation*}
We have the relation
\begin{equation*}
\begin{split}
  (\iota_{\hat{\partial}_a}\iota_{\hat{\partial}_b}\Volg) \wedge dx^c
  &= \iota_{\hat{\partial}_a}[(\iota_{\hat{\partial}_b}\Volg) \wedge dx^c] 
  - (-1)^{n-1}(\iota_{\hat{\partial}_b}\Volg) \wedge 
  (\iota_{\hat{\partial}_a} dx^c )
  \\
  &= [\iota_{\hat{\partial}_a}(-1)^{n-1}\delta^c_b \Volg ] 
  - (-1)^{n-1}(\iota_{\hat{\partial}_b}\Volg) \delta^c_a
  \\
  &= (-1)^{n-1}(
    \delta^c_b \iota_{\hat{\partial}_a} 
  - \delta^c_a \iota_{\hat{\partial}_b} 
  ) \, \Volg
  \,.
\end{split}
\end{equation*}
Moreover, since $\chi^{ab} = -\chi^{ba}$, we have
\begin{equation*}
\begin{split}
  \nabla_a \chi^{ab}
  &= \Lie_{\hat{\partial}_a} \chi^{ab}
  + \Gamma^a_{ad} \chi^{db} + \Gamma^b_{ad} \chi^{ad}
  \\
  &= \Lie_{\hat{\partial}_a} \chi^{ab}
  + \Gamma^a_{ad} \chi^{db}
  \,.
\end{split}
\end{equation*}
Using these relations, we can compute the horizontal differential of $\chi$ as
\begin{equation*}
\begin{split}
  d\chi 
  &= \tfrac{1}{2}(-1)^{p+n-2} 
  (\Lie_{\hat{\partial}_c} \chi) \wedge dx^c
  \\
  &= \tfrac{1}{2}(\Lie_{\hat{\partial}_c} \chi^{ab}) \wedge 
  (-1)^{n-2}(
  \iota_{\hat{\partial}_a}\iota_{\hat{\partial}_b} \Volg) \wedge dx^c
  \\
  &\quad{}
  + \tfrac{1}{2}\chi^{ab} \wedge (-1)^{n-2}\Lie_{\hat{\partial}_c}(
  \iota_{\hat{\partial}_a}\iota_{\hat{\partial}_b} \Volg) \wedge dx^c
  \\
  &= \tfrac{1}{2}(\Lie_{\hat{\partial}_c} \chi^{ab} 
  + \chi^{ab} \Gamma^d_{dc}) \wedge 
  (-1)^{n-2}(\iota_{\hat{\partial}_a}\iota_{\hat{\partial}_b}\Volg) \wedge dx^c
  \\
  &= \tfrac{1}{2}(\Lie_{\hat{\partial}_c} \chi^{ab} 
  + \chi^{ab} \Gamma^{d}_{dc}) \wedge
  (-1)(    
  \delta^c_b \iota_{\hat{\partial}_a} 
  - \delta^c_a \iota_{\hat{\partial}_b} 
  )\Volg
  \\
  &= (\Lie_{\hat{\partial}_a} \chi^{ab}  
  + \chi^{ab} \Gamma^{d}_{da} ) \wedge
   \iota_{\hat{\partial}_b} \Volg
  \\
  &= (\nabla_a \chi^{ab}) \wedge \iota_{\hat{\partial}_b}\Volg
  \,,
\end{split}    
\end{equation*}
which finishes the proof.
\end{proof}

\section{The homotopy momentum map of general relativity}
\label{sec:GRmomentum}

We now have all the tools needed for the multisymplectic interpretation of the diffeomorphism symmetry of general relativity. We start by recalling the Euler-Lagrange and the standard boundary form. Then we show in Thm.~\ref{thm:LepageInvariant} that the Lepage form is invariant under the diagonal action of vector fields. In other words, the action of vector fields is a manifest diffeomorphism symmetry of general relativity in the sense of Def.~\ref{def:ManifestDiffSym}. It follows from Prop.~\ref{prop:ManifestMomMap} that the symmetry has a homotopy momentum map, which is given explicitly in Thm.~\ref{thm:GRHomMomentum}.

\subsection{Euler-Lagrange and boundary form}
\label{eq:LepageGR}

The lagrangian form of the Hilbert-Einstein action is
\begin{equation}
\label{eq:HilbEinstForm}
  L = R \,\Volg \,,
\end{equation}
where $R$ is the scalar curvature, which has to be interpreted within the variational bicomplex as a function on $J^\infty \Lor$ as follows: The Riemann curvature tensor is given in local coordinates in terms of the connection coefficients~\eqref{eq:ConnectionCoeffs} by
\begin{equation*}
  \Riem_{abc}{}^d
  = \hat{\partial}_{b} \Gamma^d_{ac}
  - \hat{\partial}_{a} \Gamma^d_{bc}
  + \Gamma^e_{ac} \Gamma^d_{eb}
  - \Gamma^e_{bc} \Gamma^d_{ea}
  \,.
\end{equation*}
This is the usual formula \cite[Eq.~(3.4.4)]{Wald:GR} with the partial coordinate derivatives replaced by the Cartan lifts $\hat{\partial}_a$ and $\hat{\partial}_b$. The Ricci curvature is given by the contraction $\Ric_{ab} := \Riem_{aeb}{}^e$ and the scalar curvature by the trace of the Ricci curvature $R = g^{ab}\Ric_{ab}$.

The vertical differential of the scalar curvature $R = g^{ab} \Ric_{ab}$ is given by
\begin{equation*}
  \delta (g^{ab} \Ric_{ab})
  = \delta g^{ab} \Ric_{ab} + g^{ab} \delta \Ric_{ab}
  \,.
\end{equation*}
The first term can be written as
\begin{equation*}
  \delta g^{ab} \Ric_{ab}
  = - \Ric^{ab} \delta g_{ab}
\end{equation*}
The second term is given by \cite[Eq.~(E.1.15)]{Wald:GR}
\begin{equation*}
  g^{ab} \delta \Ric_{ab} = \nabla^a \gamma_a
\end{equation*}
where
\begin{equation*}
  \gamma_a = g^{bc}(\nabla_c \delta g_{ab} 
  - \nabla_a \delta g_{bc})
  \,,
\end{equation*}
and where the covariant derivative is to be understood as
\begin{equation*}
\begin{split}
  \nabla^a \gamma_a 
  &= g^{ab} (\Lie_{\hat{\partial}_a} \gamma_b 
    - \Gamma^{c}_{ab} \gamma_{c} )
  \,,
\end{split}
\end{equation*}
as explained in Sec.~\ref{sec:CovDerForms}. The vertical differential of the volume form was computed in Eq.~\eqref{eq:deltaVol}. Putting everything together, we get
\begin{equation*}
\begin{split}
  \delta L
  &= - \bigl( \Ric^{ab} - \tfrac{1}{2} R g^{ab} \bigr)\delta g_{ab}
  \wedge \Volg
  + \nabla^a \gamma_a \wedge \Volg
  \,.
\end{split}
\end{equation*}
The first term is the Euler-Lagrange form
\begin{equation*}
  EL = - G^{ab} \delta g_{ab} \wedge \Volg
  \,,
\end{equation*}
where
\begin{equation*}
  G^{ab} = \Ric^{ab} - \tfrac{1}{2} R g^{ab}
\end{equation*}
is the Einstein tensor. The Einstein tensor is divergence-free, i.e.
\begin{equation*}
\begin{split}
  \nabla_a G^{ab}
  &= \Lie_{\hat{\partial}_a} G^{ab}
  + \Gamma^{a}_{ac} G^{cb} + \Gamma^{b}_{ac} G^{ac}
  \\
  &= 0
  \,.
\end{split}  
\end{equation*}
Using Eq.~\eqref{eq:dDiv1}, the second term can be written as a $d$-exact term
\begin{equation*}
  (\nabla^a \gamma_a) \wedge \Volg 
  = - d\gamma
  \,,
\end{equation*}
where
\begin{equation}
\label{eq:BoundarForm}
\begin{split}
  \gamma 
  &= \gamma^a \wedge \iota_{\hat{\partial}_a} \Volg
  \\
  &= g^{ad} g^{bc}(\nabla_c \delta g_{ab} 
  - \nabla_a \delta g_{bc})
  \wedge \iota_{\hat{\partial}_d} \Volg
\end{split}
\end{equation}
is the boundary form.


\subsection{Invariance of the Lepage form}

\begin{Theorem}
\label{thm:LepageInvariant}
The Lepage form $L + \gamma$ given by the sum of the Hilbert-Einstein lagrangian~\eqref{eq:HilbEinstForm} and the boundary form~\eqref{eq:BoundarForm} is invariant under the action~\eqref{eq:DiffActionGR} of spacetime vector fields. In other words, the action is a manifest diffeomorphism symmetry in the sense of Def.~\ref{def:ManifestDiffSym}.
\end{Theorem}
\begin{proof}
The invariance must hold independently in every bidegree, so that we need to prove the two equations 
\begin{equation*}
  \Lie_{\xi_v + \hat{v}} L = 0 \,,\qquad
  \Lie_{\xi_v + \hat{v}} \gamma = 0 \,.
\end{equation*}
We start by proving the invariance of $L$. We have
\begin{equation}
\label{eq:LepagInv1}
\begin{split}
  \Lie_{\xi_v} L
  &= \iota_{\xi_v} \delta L
  = \iota_{\xi_v} (EL - d\gamma)
  \\
  &= \iota_{\xi_v} EL + d\iota_{\xi_v}\gamma
  \,.
\end{split}    
\end{equation}
We will compute both summands separately. First we use~\eqref{eq:xivEvol} to compute
\begin{equation*}
\begin{split}
  \iota_{\xi_v} \delta g_{ab}
  &= 
  - \Bigl( v^c g_{ab,c} 
  + \frac{\partial v^{a'}}{\partial x^a} g_{a'b}
  + \frac{\partial v^{b'}}{\partial x^b} g_{ab'}
  \Bigr)
  \\
  &= 
  - \bigl( v^c g_{ab,c} 
  + \partial_{\hat{\partial}_a}(v^c g_{cb}) - v^c g_{cb,a} 
  + \partial_{\hat{\partial}_b}(v^c g_{ac}) - v^c g_{ac,b} 
  \bigr)
  \\
  &= 
  - \bigl(
    \partial_{\hat{\partial}_a}v_b 
  + \partial_{\hat{\partial}_b}v_a
  - v_c g^{ce}(g_{ca,b} + g_{cb,a} - g_{ab,c})
  \bigr)
  \\
  &= 
  - \bigl(
    \partial_{\hat{\partial}_a}v_b 
  + \partial_{\hat{\partial}_b}v_a
  - v_c 2 \Gamma^c_{ab}
  \bigr)
  \\
  &= 
  - ( \nabla_a v_b + \nabla_b v_a )
  \,,
\end{split}
\end{equation*}
where we have used~\eqref{eq:ConnectionCoeffs}. With this formula we obtain
\begin{equation}
\label{eq:LepagInv2}
\begin{split}
  \iota_{\xi_v} EL
  &= G^{ab}(\nabla_a v_b + \nabla_b v_a) \Volg
  \\
  &= 2\bigl(\nabla_a (G^{ab} v_b)\bigr)\, \Volg
  \\
  &= d\bigl( 2 G^{ab} v_b \iota_{\hat{\partial}_a} \Volg \bigr)
  \,,
\end{split}    
\end{equation}
where in the last step we have used the divergence formula~\eqref{eq:dDiv1}. For the second term we compute
\begin{equation}
\label{eq:LepagInv3}
\begin{split}
  \iota_{\xi_v} \gamma 
  &= [\iota_{\xi_v} g^{ad} g^{bc}(\nabla_c \delta g_{ab} 
  - \nabla_a \delta g_{bc})] \wedge
  \iota_{\hat{\partial}_d} \Volg 
  \\
  &= g^{ad}g^{bc}[- \nabla_c(\nabla_a v_b + \nabla_b v_a) 
  + \nabla_a (\nabla_b v_c + \nabla_c v_b)] \,
  \iota_{\hat{\partial}_d} \Volg 
  \\
  &= g^{ad}g^{bc}[\nabla_c\nabla_a v_b - \nabla_c \nabla_b v_a 
  - 2 \nabla_c \nabla_a v_b
  + \nabla_a (\nabla_b v_c + \nabla_c v_b)] \,
  \iota_{\hat{\partial}_d} \Volg 
  \\
  &= [\nabla_c(\nabla^d v^c - \nabla^c v^d) 
  - 2 g^{ad} g^{bc}( \nabla_c \nabla_a - \nabla_a \nabla_c) v_b] \,
  \iota_{\hat{\partial}_d} \Volg 
  \\
  &= [\nabla_c(\nabla^d v^c - \nabla^c v^d)
  - 2 \Ric^{bd} v_b ] \,
  \iota_{\hat{\partial}_d} \Volg 
  \\
  &= - 2\Ric^{ab} v_a \,\iota_{\hat{\partial}_b} \Volg
  + 
  [\nabla_a (\nabla^b v^a - \nabla^a v^b) ] \, 
  \iota_{\hat{\partial}_b} \Volg
  \\
  &= 
  - 2\Ric^{ab} v_a \,\iota_{\hat{\partial}_b} \Volg
  - d \bigl( \tfrac{1}{2} (\nabla^a v^b - \nabla^b v^a) \, 
  \iota_{\hat{\partial}_a}\iota_{\hat{\partial}_b} \Volg
  \bigr)
  \,,
\end{split}
\end{equation}
where in the last step we have used the divergence formula~\eqref{eq:dDiv2}. Inserting~\eqref{eq:LepagInv2} and~\eqref{eq:LepagInv3} into the right hand side of~\eqref{eq:LepagInv1}, we obtain
\begin{equation*}
\begin{split}
  \Lie_{\xi_v} L
  &=
  2 d\bigl(
    G^{ab} v_b \iota_{\hat{\partial}_a} \Volg
    - 2\Ric^{ab} v_a \,\iota_{\hat{\partial}_b} \Volg
    \bigr)
  \\
  &=
  - d \bigl( R v^a \iota_{\hat{\partial}_a} \Volg \bigr)
  \\
  &= - \Lie_{\hat{v}} L
  \,,
\end{split}
\end{equation*}
which finishes the proof of the invariance of $L$.

It remains to prove the invariance of $\gamma$. The strategy of the proof is to show that all indices appearing in
\begin{equation*}
  \gamma 
  = g^{ad} g^{bc}(\nabla_c \delta g_{ab} 
  - \nabla_a \delta g_{bc})
  \wedge \iota_{\hat{\partial}_d} \Volg
\end{equation*}
are covariant or contravariant in the sense of Def.~\ref{def:CovariantForms}, so that their contraction is invariant by Lem.~\ref{lem:ContractionInvariant}.

We have shown in Lem.~\ref{lem:VolInvariant} that the volume form is invariant. It follows from Lem.~\ref{lem:InnnerDerConvariant} that the index $d$ of $\iota_{\hat{\partial}_d} \Volg$ is covariant. We have shown in Eq.~\eqref{eq:Lieginv} that the indices of $g^{ad}$ and $g^{bc}$ are contravariant. In Eq.~\eqref{eq:Liedeltag} we have seen that the indices of $\delta g_{bc}$ are covariant. It follows from Lem.~\ref{lem:NablaCovariant} that the indices of the covariant derivatives $\nabla_c$ and $\nabla_a$ are covariant. Lem.~\ref{lem:WedgeCovariant} shows that the wedge product is contravariant in all upper and covariant in all lower indices. With Lem.~\ref{lem:ContractionInvariant} we conclude that $\gamma$ is invariant.
\end{proof}

\begin{Theorem}
\label{thm:GRHomMomentum}
The action of spacetime vector fields on the infinite jet bundle of Lorentz metrics defined in~\eqref{eq:DiffActionGR} has a homotopy momentum map
\begin{equation*}
  \mu: \calX(M) \longrightarrow 
  L_\infty(J^\infty \Lor, EL + \delta\gamma)
  \,,
\end{equation*}
given by
\begin{equation*}
\begin{aligned}
  \mu_k: \wedge^k \calX(M)
  &\longrightarrow L_\infty(J^\infty \Lor, EL + \delta\gamma)
  \\
  \mu_k(v_1, \ldots, v_k) 
  &:= \iota_{\rho(v_1)} \cdots \iota_{\rho(v_k)} (L + \gamma)
  \,.  
\end{aligned}
\end{equation*}
\end{Theorem}
\begin{proof}
The proof follows from Thm.~\ref{thm:LepageInvariant} and Prop.~\ref{prop:ExactMomentum}.
\end{proof}

The Noether current, which was given in~\eqref{eq:NoetherCurrLFT} by the general formula $j_v = - \iota_{\hat{v}} L - \iota_{\xi_v} \gamma$, can be computed with~\eqref{eq:LepagInv3} to 
\begin{equation}
\label{eq:NoethCurrGR}
  j_v 
  = 2 G^{ab} v_a \wedge \iota_{\hat{\partial}_b} \Volg
  + d \bigl( \tfrac{1}{2} (\nabla^a v^b - \nabla^b v^a) \, 
  \iota_{\hat{\partial}_a}\iota_{\hat{\partial}_b} \Volg
  \bigr)
  \,.
\end{equation}
The $k=1$ component of the homotopy momentum map, which was given in~\eqref{eq:mu1} by the general formula $\mu_1(v) = - j_v + \iota_{\hat{v}}\gamma$, is
\begin{equation*}
\begin{split}
  \mu_1(v) 
  &= 
  - 2 G^{ab} v_a \wedge \iota_{\hat{\partial}_b} \Volg
  - d \bigl( \tfrac{1}{2} (\nabla^a v^b - \nabla^b v^a) \, 
  \iota_{\hat{\partial}_a}\iota_{\hat{\partial}_b} \Volg
  \bigr)
  \\
  &{}\quad
 + g^{ad} g^{bc}(\nabla_c \delta g_{ab} 
  - \nabla_a \delta g_{bc}) v^e 
    \wedge \iota_{\hat{\partial}_d} \iota_{\hat{\partial}_e} \Volg
  \,.
\end{split}    
\end{equation*}

\begin{Remark}
\label{rmk:NotEnergyMomentum}
The Noether current of a symmetry is determined only up to a $d$-closed form. Usually, the second summand of~\eqref{eq:NoethCurrGR} is dropped, so that the Noether current is $C^\infty(M)$-linear in $v$ and can be interpreted as the energy-momentum tensor $G^{ab}$. Here, we must take~\eqref{eq:NoethCurrGR} as Noether current so that $\mu$ is a homomorphism of $L_\infty$-algebras. 
\end{Remark}

\subsection*{Acknowledgements} This paper was branched out of a long and ongoing collaboration with Michele Schiavina and Alan Weinstein \cite{BlohmannSchiavinaWeinstein:2022} on the constraint problem of general relativity. They have contributed with invaluable discussions and encouraged me to publish the homotopical approach separately. I am indebted to Ya\"el Fr\'egier, Chris Rogers, and Marco Zambon for teaching me about homotopy momentum maps and for many illuminating discussions over the years. Finally, I would like to thank Janina Bernardy and Leonard Hofmann for valuable feedback on various versions of this paper. 

\bibliographystyle{alpha}
\bibliography{HomMom}

\end{document}